\documentclass[12pt]{article} 
\usepackage{graphicx}
\usepackage{amssymb}
\usepackage{mathtools}
\usepackage{amsmath}
\usepackage{amsthm}
\usepackage[hidelinks]{hyperref}
\usepackage{csquotes}
\usepackage[pagewise]{lineno}
\usepackage[english]{babel}
\usepackage{enumitem}
\usepackage[utf8]{inputenc} 
\usepackage{sectsty}
\usepackage{tocloft}

\newcommand{\bigzero}{\mbox{\normalfont\large 0}}
\renewcommand{\Re}{\operatorname{Re}}
\newtheorem{thm}{Theorem}[subsection]
\newtheorem{cor}{Corollary}[subsection]
\newtheorem{lem}[thm]{Lemma}
\newtheorem{prop}[thm]{Proposition}
\newtheorem{defn}{Definition}[subsection]
\newtheorem{exmp}{Example}[subsection]
\newtheorem{rmk}{Remark}[subsection]
\numberwithin{equation}{section}
\sectionfont{\fontsize{13}{15}\selectfont}
\subsectionfont{\fontsize{12}{15}\selectfont}

\def \R{{\mathbb{R}}}
\def \P {\Phi(x)}
\def \japxi{\langle \xi \rangle}
\def \japxin{\langle \xi \rangle_k}
\def \japx{\langle x \rangle}
\def \hyp{Z_{ext}(N)}
\def \pd{Z_{int}(N)}
\def \var{(t,x,\xi)}
\def \J{[0,T] \times \R^n \times \R^n}
\def \la{\langle}
\def \ra{\rangle}
\def \h{h(x,\xi)}
\def \ran{\rangle_k}

\providecommand{\keywords}[1]
{
	\small	
	\textbf{\text{Keywords:}} #1
}

\providecommand{\subclass}[1]
{
	\small	
	\textbf{\text{MSC (2010):}} #1
}

\title{\Large Global Well-Posedness of a Class of Strictly Hyperbolic Cauchy Problems with Coefficients Non-Absolutely Continuous in Time\thanks{Authors dedicate this work to Bhagawan Sri Sathya Sai Baba.}}

\author{\normalsize Rahul Raju Pattar\thanks{rahulrajupattar@sssihl.edu.in (Corresponding Author)} , N. Uday Kiran\thanks{nudaykiran@sssihl.edu.in}  \\ \small Department of Mathematics and Computer Science\\	\small Sri Sathya Sai Institute of Higher Learning, Puttaparthi, India \\
}
\date{}
\begin{document} 
	
		\maketitle	
		
			\begin{abstract}
			We investigate the behavior of the solutions of a class of certain strictly hyperbolic equations defined on $[0,T]\times \R^n$ in relation to a class of metrics on the phase space. In particular, we study the global regularity and decay issues of the solution to an equation with coefficients polynomially bound in $x$ and with their $t$-derivative of order $\textnormal{O}(t^{-q}),$ where $q \in \big[1,\frac{3}{2}\big)$. For this purpose, an appropriate generalized symbol class based on the metric is defined and the associated Planck function is used to define a new class of infinite order operators to perform conjugation. We demonstrate that the solution not only experiences a loss of regularity (usually observed for the case of coefficients bounded in $x$) but also a decay in relation to the initial datum defined in a Sobolev space tailored to the generalized symbol class. Further, we observe that a precise behavior of the solution could be obtained by making an optimal choice of the metric in relation to the coefficients of the given equation. We also derive the cone conditions in the global setting.\\
			\\
			\\
			\subclass{35L30 $\cdot$ 35S05 $\cdot$ 35B65 $\cdot$ 35B30}
		\end{abstract}
	
	\tableofcontents
	
		\section{Introduction}\label{intro}
	
	Loss of regularity (or derivatives) plays a prominent role in the modern theory of Hyperbolic Cauchy Problems. In certain cases, it has been observed that the solution of a hyperbolic equation experiences a loss of regularity in relation to the initial datum when the coefficients are either irregular or degenerate (see for instance \cite{Cico1,CicoLor,LorReis,Petkov}, and the references therein). There is an extensive literature on the topic of loss of regularity in hyperbolic equations; nevertheless, an optimal estimation of the loss is a far fetched goal. In this paper, using a class of metrics on the phase space, we aim to address the behavior of the solution in relation to the $x$ dependency of the coefficients. 
	
	Metrics on the phase space are now widely used in the pseudodifferential operator theory to address the solvability and regularity issues \cite{nicRodi,Lern}. Last decade has seen the complete resolution of the Nirenberg-Treves conjecture \cite{Dencker} and more recently in obtaining the loss of derivatives in the Ivrii-Petkov conjecture \cite{Petkov,Nishi} - thanks to the metric on the phase space for both these achievements. Furthermore, global issues in the pseudodifferential theory require a direct application of the metrics on phase space. 
	
	The notion of a metric on phase space $T^{*}\R^{n}\ (\cong \R^{2n})$ was first introduced by L. H\"ormander \cite{Horm} who studied the smooth functions $a(x,\xi)$ called symbols, using the metric $\japxi^{2\delta} |dx|^2+ \japxi^{-2\rho} |d\xi|^2$, $0\leq \delta<\rho\leq 1$. That is, the symbol $a(x,\xi)$ satisfies, for some $m \in \R \text{ and } C_{\alpha\beta}>0$,
	\begin{linenomath*}
		\[
		|\partial_\xi^\alpha \partial_x^\beta a(x,\xi)| \leq C_{\alpha\beta} \japxi^{m-\rho|\alpha|+\delta|\beta|}, 
		\]
	\end{linenomath*}
	for all multi-indices $\alpha,\beta \in \mathbb{N}^n$ and $\japxi=(1+|\xi|^2)^{1/2}$. In a more general framework created by R. Beals and R. Fefferman \cite{Feff,nicRodi}, we can consider the symbol $a(x,\xi)$ that satisfies the estimate for some $C_{\alpha\beta}>0,$
	\begin{linenomath*}
		\begin{equation*}		
			|\partial_\xi^\alpha \partial_x^\beta a(x,\xi)| \leq C_{\alpha\beta} \Psi(x,\xi)^{m_1-|\alpha|} \Phi(x,\xi)^{m_2-|\beta|},
		\end{equation*}
	\end{linenomath*}
	where $m_1,m_2 \in \R$ and the positive functions $\Psi(x,\xi)$ and $\Phi(x,\xi)$ are specially chosen. In such a case, as given in \cite{Lern}, we can consider the following Riemannian structure on the phase space 
	\begin{linenomath*}
		\begin{equation}
			\label{beals_fefferman}
			g_{x,\xi}=\frac{|dx|^2}{\Phi(x,\xi)^2}+\frac{|d\xi|^2}{\Psi(x,\xi)^2};
		\end{equation}
	\end{linenomath*}
	thus, one has geometric restrictions on $\Phi(x,\xi)$ and $\Psi(x,\xi)$ of both Riemannian type and Symplectic type \cite{Lern}. In this work, we consider a metric (\ref{beals_fefferman}) with $\Phi(x,\xi)=\Phi(x)$ and $\Psi(x,\xi)=\japxin=(k^2+|\xi|^2)^{1/2}$, for a large parameter $k$. We discuss this class in Section \ref{struct}.  
	
	In this paper, we obtain a global well-posedness result (Theorem \ref{result2}) on a class of strictly hyperbolic equations that have a well-known behavior of a loss of derivatives in the local setting. Along with the extension of the results to a global setting, we also investigate the behavior at infinity of the solution in relation to the coefficients. By a loss of derivatives of the solution in relation to the initial datum we mean a change of indices in an appropriate Sobolev space associated with the metric (\ref{beals_fefferman}). 
	
	To discuss the results from literature on the well-posedness of a Cauchy problem with coefficients low-regular in $t$, we consider the simple case: 
	\begin{linenomath*}
		\begin{equation}
			\label{model}
			u_{tt}-a(t,x)u_{xx}=0, \;\; u(0,x)=g_1(x) \;\textnormal{ and }\; u_t(0,x)=g_2(x).
		\end{equation}
	\end{linenomath*}
	The well-posedness results in Sobolev spaces for the case of Lipschitz regularity in $t$ are already familiar (see \cite[Chapter 9]{Horm1}). In the irregular coefficient category, broadly speaking, there are two approaches which have been employed to weaken the Lipschitz regularity on $a(t,x)$: modulus of continuity based (cf. \cite{CicoLor,AscaCappi1,AscaCappi2}) and non-modulus of continuity based. The latter approach can be based on singular behaviour of type $t^{-q}$, $q \geq 1$, as $t \to 0$, of the first $t$-derivative (cf. \cite{CSK,Cico1,Cico2}), oscillatory behaviour (cf. \cite{KuboReis,KinoReis} ) or low-regularity in second variation with respect to $t$ (cf. \cite{CSFM}).
	
	In \cite{CSK}, Colombini et al. have discussed the well-posedness of (\ref{model}) when $a(t,x)$ is independent of $x$ and admitting an $O(t^{-q}),q\geq 1$, singularity at $t=0$, in the first $t$-derivative. These results were extended by Cicognani \cite{Cico1} to the case of coefficients depending on $x$ as well and satisfying 
	\begin{linenomath*}
		\[
		\vert \partial_t \partial_x^\beta a(t,x)\vert \leq \frac{C_\beta}{t^q}, \quad q \geq 1, \: (t,x) \in (0,T] \times \R^n,
		\] 	
	\end{linenomath*}
	where $a \in C^1((0,T];\mathcal{B}(\R^n))$ $\Big( C([0,T];G^\sigma(\R^n))\cap C^1((0,T];G^\sigma(\R^n))\;$ where $1 < \sigma < q/(q-1) \Big)$ for $q=1(>1, \text{ respectively})$. $\mathcal{B}(\R^n)$ denotes the space of all $C^\infty$ functions that are bounded together with all their derivatives in $\R^n$ and $G^\sigma(\R^n)$ denotes Gevrey space of index $\sigma$. Here the author reports well-posedness in $C^\infty(\R^n)( G^\sigma(\R^n))$ for 
	the Cauchy problem (\ref{model}) with a finite (resp. infinite) loss of derivatives for $q=1$ (resp. $>1$). 
	
	The global results for the SG-hyperbolic systems with the log-Lipschitz and H\"older regularity assumption on $t$  have been studied by Ascanelli and Cappiello in \cite{AscaCappi1} and \cite{AscaCappi2}, respectively. With regard to coefficients with oscillatory behaviour, the second author along with Coriasco and Battisti \cite{NUKCori1} have extended the work of Kubo and Reissig \cite{KuboReis} to the SG setting, using the generalized parameter dependent Fourier integral calculus. The authors consider a weakly hyperbolic Cauchy problem with 
	\begin{linenomath*}
		\[
		\vert \partial_t^l \partial_x^\beta a(t,x)\vert \leq C_{l,\beta} \bigg(\frac{1}{t}\bigg(\ln \frac{1}{t}\bigg)^\gamma\bigg)^l \japx^{2-\vert \beta \vert}, 
		\]
	\end{linenomath*}
	where $(t,x) \in (0,T] \times \R^n, \; 0 \leq \gamma \leq 1,$ for all $l\in \mathbb{N}\textnormal{ and }\beta$ multi-index, for the model problem (\ref{model}). The global extension results not only indicate a loss of regularity in the solution but also display a loss in the decay.
	
	In this paper, we consider a strictly hyperbolic equation of $m^{th}$ order and investigate the role played by the nature of the coefficients being continuous in $t$ and admitting a singularity of order $O(t^{-q}),\;1 \leq q < \frac{3}{2}$, on the $t$-derivative. In the $x$ variable, we consider an optimal $\P$ such that the following holds:
	\begin{linenomath*}
		\begin{equation}
			\label{3}
			C^{-1} \P^2  \leq |a(t,x)| \leq C \P^2,
		\end{equation}  
	\end{linenomath*}
	where $1<C$ and $1 \leq \P \lesssim \japx$ for an optimal choice of $C$ and $\Phi(x)$ as discussed in Section \ref{struct}. Along with the condition (\ref{3}), in our study, we consider the following metric on the phase space:
	\begin{linenomath*}
		\begin{equation}
			\label{om}
			g_{\Phi,k} = \Phi(x)^{-2}\vert dx \vert^2 + \japxin^{-2} \vert d\xi \vert^2.
		\end{equation}
	\end{linenomath*}   
	We develop an operator calculus related to the metric $g_{\Phi,k}$ (see Appendix II and III, also see \cite[Chapter 1 and 6]{nicRodi}). 
	
	A special feature of our calculus is the explicit appearance of the Planck function. In fact, based on the Planck function we define a test function space $\mathcal{M}^\sigma_{\Phi,k}(\R^n)$ (see Definition (\ref{mspace})) which in turn is fundamental in the definition our class of finite and infinite order operators.  
	
	One of the key steps in our analysis that helps in dealing with low-regularity in $t$ is the conjugation by a new class of infinite order pseudodifferential operators of the form $e^{\Lambda(t)(\P\la D_x \ran)^{1/\sigma}}$, where $\Lambda(t)$ is a continuous function for $t\in [0,T], T>0$ and $\la D_x\ra_k=(k^2-\Delta_x)^{1/2}$. Here, the symbol of the operator $\P\la D_x \ran$ is given by $h(x,\xi)^{-1}=\P\japxin$ where $\h$ is the Planck function related to the metric $g_{\Phi,k}$ in (\ref{om}). Note that, in the literature (see \cite{AscaCappi2}), some authors have used an infinite order pseudodifferential operators of the form $e^{\Lambda(t)(\japx^{1/\sigma}+\la D \ra^{1/\sigma})}$ for conjugation. We report (see Theorem \ref{conju}) that the metric governing the decay estimates of the lower order terms changes after the conjugation. In our case, this metric is of the form
	\begin{linenomath*}
		\begin{equation} \label{om2}
			\tilde{g}_{\Phi,k} = \bigg(\frac{\japxin^{\frac{1}{\sigma}}}{\P^\gamma}\bigg)^2|dx|^2 + \bigg(\frac{\P^{\frac{1}{\sigma}}}{\japxin^\gamma}\bigg)^2 |d\xi|^2,
		\end{equation}
	\end{linenomath*}
	where $\gamma=1-\frac{1}{\sigma}$. We demonstrate this in Section \ref{Conj}.
	
	The Sobolev space is defined using the infinite order pseudodifferential operator $e^{\Lambda(t)(\P\la D_x \ran)^{1/\sigma}}$ and we prove the well-posedness in the scale of these spaces with an appropriate choice of the function $\Lambda(t)$. Furthrtmore, in Section \ref{cone}, we derive an optimal cone condition for the solution of the Cauchy problem from the energy estimate used in proving the well-posedness.
	
			\section{The Main Result}  
	Before we state the main result of this paper, we introduce our model $m^{th}$ order strictly hyperbolic equation and define an appropriate Sobolev space for our purpose. 
	
	\subsection{Our Model of Strictly Hyperbolic Operator}
	We deal with the Cauchy problem 
	\begin{linenomath*}
		\begin{equation}
			\begin{cases}
				\label{eq1}
				P(t,x,\partial_t,D_x)u(t,x)= f(t,x), \qquad (t,x) \in [0,T] \times \R^n, \\
				\partial_t^{j-1}u(0,x)=f_j(x), \qquad \qquad \qquad \; j=1,\dots,m
			\end{cases}
		\end{equation}
	\end{linenomath*}
	where the operator $P(t,x,\partial_{t},D_{x})$ is given by
	\begin{linenomath*}		
		\begin{align*}
			P=\partial_t^m-\sum_{j=0}^{m-1}\Big(A_{m-j}(&t,x,D_x)+B_{m-j}(t,x,D_x)\Big)\partial_t^j \quad \text{ with }\\
			A_{m-j}(t,x,D_x)  &=\sum_{\vert \alpha \vert+j=m}a_{j,\alpha}(t,x)D_x^\alpha \quad \text{ and } \\
			B_{m-j}(t,x,D_x) &=\sum_{\vert \alpha\vert+j<m}b_{j,\alpha}(t,x)D_x^\alpha,
		\end{align*}
	\end{linenomath*}
	using the usual multi-index notation we denote $D_x^\alpha={(-i)}^\alpha\partial_x^\alpha$. 
	The operator $P$ in (\ref{eq1}) is said to be strictly hyperbolic operator if  the symbol of the principal part	
	\begin{linenomath*}
		\begin{align*}
			P_m(t,x,i\tau,\xi) &={(i\tau)}^m-\sum_{j=0}^{m-1}A_{m-j}(t,x,\xi){(i\tau)}^j \\
			&= {(i\tau)}^m-\sum_{j=0}^{m-1}\sum_{\vert \alpha\vert+j=m}a_{j,\alpha}(t,x)\xi^\alpha{(i\tau)}^j
		\end{align*}
	\end{linenomath*}
	has purely imaginary characteristic roots $i\tau_j(t,x,\xi),\: j=1,\dots,m$ where $\tau_j(t,x,\xi)$ is a real-valued, simple function in $t$ and positively homogeneous of degree $1$ for $\xi \not=0$ in $\mathbb{R}^{n}$. These roots are numbered and arranged so that \begin{linenomath*}
		\begin{align}
			\tau_1(t,x,\xi)<\tau_2(t,x,\xi)&<\cdots<\tau_m(t,x,\xi), \quad \text {and} \\
			\label{roots}
			C \P \japxi &\leq  \vert \tau_j(t,x,\xi) \vert,
			%\tau_{k+1}(t,x,\xi)-\tau_{k}(t,x,\xi) &> C_2 \Phi(x) \japxi, 
		\end{align}
	\end{linenomath*}
	for some $C=C(t)>0$, for all $t\in[0,T],\: x,\xi \in \R^n$ and $1 \leq j \leq m$. 	
	
	We assume that the coefficients are non-absolutely continuous in $t$ and this low-regular behaviour of the coefficients $a_{j,\alpha}$ corresponding to the top order terms (i.e., $|\alpha|+j=m$) is characterized by singular behavior of order $O(t^{-q}),\; 1\leq q < \frac{3}{2}$, at $t = 0$ of their first $t$-derivatives. We assume $a_{j,\alpha} \in C([0,T];C^\infty(\R^n)) \cap C^1((0,T];C^\infty(\R^n))$ satisfy
	\begin{linenomath*}	
		\begin{align}
			\label{bound}
			\vert D_x^\beta a_{j,\alpha}(t,x) \vert &\leq C^{|\beta|} \beta!^\sigma \P^{m-j-\vert \beta \vert},	\qquad (t,x) \in [0,T] \times \R^n, \\
			\label{B-up2}
			\vert D_x^\beta \partial_t a_{j,\alpha}(t,x) \vert &\leq C^{|\beta|} \beta!^\sigma  \P^{m-j-\vert \beta \vert}\frac{1}{t^q}, \quad (t,x) \in (0,T] \times \R^n,
		\end{align} 
	\end{linenomath*}
	for $3 \leq \sigma < q/(q-1)$ and the coefficients of the lower order terms, $b_{j,\alpha} \in C^1([0,T];C^\infty(\R^n))$ satisfy
	\begin{linenomath*}
		\begin{equation}
			\label{Lower}
			\vert D_x^\beta b_{j,\alpha}(t,x) \vert \leq C^{|\beta|} \beta!^\sigma \P^{m-j-1-\vert \beta \vert},	\quad (t,x) \in [0,T] \times \R^n,\quad C>0.
		\end{equation}
	\end{linenomath*}
	
	A simple example of a partial differential equation of the type (\ref{eq1}), for $m=2$, $(t,x) \in [0,1]\times\R$, and $\kappa \in [0,1]$ is given below. 
	\begin{exmp}
		$P=\partial_t^2-\japx^{2\kappa} \left(2+\sin \left( \japx^{1-\kappa}\right)\right)  f(t)\partial_x^2$, where $f(t)=1 + t\sin\left(\frac{1}{t^{4/3}}\right)$ for $t \in (0,1]$ and $f(0) \equiv 1$. Here $\P=\japx^\kappa$ and $q = \frac{4}{3}$.
	\end{exmp}
	
	Observe that we have assumed $3 \leq \sigma < q/(q-1)$ where as in \cite{CSK,Cico1}, it is $1 < \sigma < q/(q-1)$. The increase in the lower bound for $\sigma$ is due to two factors: $(i)$ the uncertainty principle (which applied to the metric $\tilde{g}_{\Phi,k}$ gives $\sigma>2$. This will be dealt in detail in Section \ref{struct}.), $(ii)$ the application of sharp G\r{a}rding inequality in our context dictates that $\sigma \geq 3$ (this is discussed in Section \ref{sgi}). Due to this increment in $\sigma$, we have $q \in \big[1,\frac{3}{2}\big).$
	
	\subsection{Sobolev Spaces}
	We now introduce the Sobolev space related to the metric $g_{\Phi,k}$ that is suitable for our analysis. 
	
	\begin{defn} \label{Sobo}   
		The Sobolev space $H^{s,\varepsilon,\sigma}_{\Phi,k}(\R^n)$ for $\sigma > 2$, $\varepsilon \geq 0$ and $s=(s_1,s_2) \in \R^2$ is defined as
		\begin{linenomath*}
			\begin{equation}
				\label{Sobo2}
				H^{s,\varepsilon,\sigma}_{\Phi,k}(\R^n) = \{v \in L^2(\R^n): \P^{s_2}\la D \ran^{s_1}\exp\{\varepsilon (\P\la D_x \ran)^{1/\sigma}\}v \in L^2(\R^{n}) \},
			\end{equation} 
		\end{linenomath*}
		equipped with the norm
		$
		\Vert v \Vert_{\Phi,k;s,\varepsilon,\sigma} = \Vert \Phi(\cdot)^{s_2}\la D \ran^{s_1}\exp\{\varepsilon (\Phi(\cdot)\la D \ran)^{1/\sigma}\}v \Vert_{L^2} .
		$ 
		The operator $\exp\{\varepsilon (\P\la D_x \ran)^{1/\sigma}\}$ is an infinite order pseudodifferential operator with the Fourier multiplier $\exp\{\varepsilon(\P\japxin)^{1/\sigma}\}$.		
	\end{defn}	
	The spaces $H^{s,\varepsilon,\sigma}_{\Phi,k}(\R^n)$ and $H^{-s,-\varepsilon,\sigma}_{\Phi,k}(\R^n)$ are dual to each other. Let $s'=(s_1',s_2') \in \R^2$, $\varepsilon' \geq 0$ and $\sigma' > 2$. We have that $H^{s,\varepsilon,\sigma}_{\Phi,k}(\R^n) \subset H^{s',\varepsilon',\sigma'}_{\Phi,k}(\R^n)$ if $\sigma \leq \sigma'$, $\varepsilon' \leq \varepsilon$, $s_j' \leq s_j,j=1,2$.
	\begin{defn}\label{mspace}
		The function space $\mathcal{M}^\sigma_{\Phi,k}(\R^n),\sigma\geq 3$ is a set of functions $v\in C^\infty(\R^n)$ that satisfy 
		\begin{linenomath*}
			\[
			\Vert e^{a(\P\la D_x \ran)^{1/\sigma}}v(x) \Vert_{L^2} \leq C
			\]
		\end{linenomath*}
		for some positive constants $a$ and $C$.
	\end{defn}	
	The function space $\mathcal{M}^\sigma_{\Phi,k}(\R^n)$ and its dual ${\mathcal{M}^\sigma_{\Phi,k}}'(\R^n)$ are related to the Sobolev spaces as follows 
	\begin{linenomath*}
		\begin{equation*}%\label{space}
			\mathcal{M}^\sigma_{\Phi,k}(\R^n) =  \bigcup\limits_{\varepsilon>0} \bigcap\limits_{s \in \R^2}H^{s,\varepsilon,\sigma}_{\Phi,k}(\R^n) \quad \text{ and } \quad {\mathcal{M}^\sigma_{\Phi,k}}' (\R^n)=  \bigcap\limits_{\varepsilon>0} \bigcup\limits_{s \in \R^2} H^{s,\varepsilon,\sigma}_{\Phi,k}(\R^n).
		\end{equation*}
	\end{linenomath*}
	%\lim_{\substack{\longrightarrow \\ \varepsilon \to 0}}
	%\lim_{\substack{\longleftarrow \\ \varepsilon \to 0}}
	% s \in \R^2
	
	We can relate these spaces to the Gelfand-Shilov spaces. Let us denote the Gelfand-Shilov space of indices $\mu,\nu>0$ as $\mathcal{S}^\mu_\nu(\R^n)$. We refer to \cite[Section 6.1]{nicRodi} for the definition and properties of Gelfand-Shilov spaces.
	Note that
	\begin{linenomath*}
		\[
		k^{-1/\sigma}(\P\japxin)^{1/\sigma} \leq \frac{1}{2}(\P^{2/\sigma}+\japxi^{2/\sigma}) \leq \frac{1}{2}(\japx^{2/\sigma}+\japxi^{2/\sigma}).
		\]
	\end{linenomath*}	
	Thus, we have the inclusion
	\begin{linenomath*}
		\[
		\mathcal{S}^{\frac{\sigma}{2}}_{\frac{\sigma}{2}}(\R^n) \hookrightarrow \mathcal{M}^\sigma_{\Phi,k}(\R^n).
		\]
	\end{linenomath*}
	Further, if $\P=\japx$ we have 
	\begin{linenomath*}
		\[
		\mathcal{S}^{\frac{\sigma}{2}}_{\frac{\sigma}{2}}(\R^n) \hookrightarrow \mathcal{M}^\sigma_{\P}(\R^n)\hookrightarrow \mathcal{S}^{\sigma}_{\sigma}(\R^n).
		\]
	\end{linenomath*}
	
	In the pseudodifferential calculus (see, Appendix II and III), the transposition, composition and construction of parametrix are done modulo an operator that maps ${\mathcal{M}^\sigma_{\Phi,k}}'(\R^n)$ to $\mathcal{M}^\sigma_{\Phi,k}(\R^n)$. A detailed discussion of $\mathcal{M}^\sigma_{\Phi,k}(\R^n)$ and its dual will be dealt from an abstract viewpoint in \cite{Rahul_NUK}. 
	
	\subsection{Statement of the Main Result}
	Let $e=(1,1),e_1=(1,0)$ and $e_2=(0,1)$. 
	\begin{thm}\label{result2}
		Consider the strictly hyperbolic Cauchy problem  (\ref{eq1}) satisfying the following conditions:
		\begin{enumerate}[label=\roman*)]
			\item The coefficients $a_{j,\alpha}$ of the principal part satisfy (\ref{bound})-(\ref{B-up2}) and the coefficients $ b_{j,\alpha}$ satisfy (\ref{Lower}).
			
			\item The initial data $f_k$ belongs to $H^{s+(m-k)e,\Lambda_1,\sigma}_{\Phi,k},\Lambda_1>0$ for $k=1,\cdots,m$.
			
			\item The right hand side $f \in C([0,T];H^{s,\Lambda_2,\sigma}_{\Phi,k}),\Lambda_2>0$.
		\end{enumerate}			
		Then, there exist a continuous function $\Lambda(t)$ and $\Lambda_0>0$, such that there is a unique solution
		\begin{linenomath*}
			\[
			u \in \bigcap\limits_{j=0}^{m-1}C^{m-1-j}\Big([0,T];H^{s+je,\Lambda(t),\sigma}_{\Phi,k}\Big)
			\]
		\end{linenomath*}
		for $\Lambda(t) < \min\{\Lambda_0,\Lambda_1,\Lambda_2\}$. More specifically, for a sufficiently large $\lambda$ and $\delta \in (0,1)$, we have the a-priori estimate
		\begin{linenomath*}
			\begin{equation}
				\begin{aligned}
					\label{est2}
					\sum_{j=0}^{m-1} \Vert \partial_t^ju(t,\cdot) \Vert_{\Phi,k;s+(m-1-j)e,\Lambda(t),\sigma} \; &\leq C \Bigg(\sum_{j=1}^{m} \Vert f_j\Vert_{\Phi,k;s+(m-j)e,\Lambda(0),\sigma} \\
					& \qquad + \int_{0}^{t}\Vert f(\tau,\cdot)\Vert_{\Phi,k;s,\Lambda(\tau),\sigma}\;d\tau\Bigg)
				\end{aligned}
			\end{equation}
		\end{linenomath*}
		for $0 \leq t \leq T \leq (\delta\Lambda^*/\lambda)^{1/\delta}, \; C=C_s>0$ and $\Lambda(t)=\frac{\lambda}{\delta}(T^\delta-t^\delta)$.				
	\end{thm}  	 
	The constants $\Lambda_0$ and $\delta$ in the above theorem are the constants in Theorem \ref{conju} and equation (\ref{delta}), respectively.
	\setcounter{rmk}{1}
	\begin{rmk}
		\begin{enumerate}[label=\roman*)]
			\item The Cauchy problem (\ref{eq1}) can be shown to be well-posed in $C^\infty$ for the case $q=1$ in (\ref{bound})-(\ref{B-up2}), as in \cite{Cico1}. The loss, in such a case, is finite but arbitrary. In fact, it is proportional to $\frac{T^\varepsilon}{\varepsilon}$ for an arbitrary $\varepsilon \in (0,1)$. Theorem \ref{result2} addresses the case in the Gevrey setting, $G^\sigma(\R^n)$, $3 \leq  \sigma < \infty$. Though the loss is infinite in our generalized Sobolev space setting, it is proportional to $\sigma \;t^{\frac{1}{\sigma}}$, as $\delta = 1/\sigma$ in (\ref{delta}) for $q=1$.
			
			\item For the case $q=1$, one can assume $a_{j,\alpha} \in C^1((0,T];C^\infty)$ with a global version of conditions in \cite[Theorem 1]{Cico1}. In such a case the study of global well-posedness involves $t$-dependent metric. We aim to address this case in our future work.
		\end{enumerate}  	 
	\end{rmk}

			\section{Generalized Symbol Classes Related to the Metric $g_{\Phi,k}$ }	\label{Symbol classes}
	In this section, we list the properties of $\Phi$ and define generalized global symbol classes for the pseudodifferential operators using the metrics $g_{\Phi,k}$ and $\tilde{g}_{\Phi,k}$.
	
	\subsection{Structure of $\Phi$ and its Optimal Choice}\label{struct}
	Before we discuss the structure of $\Phi$ in (\ref{om}), let us review some notation and terminology used in the study of metric on the phase space, see \cite[Chapter 2]{Lern} for details. Let us denote by $\omega(X,Y)$ the standard symplectic form on $T^*\R^n\cong \R^{2n}$: if $X=(x,\xi)$ and $Y=(y,\eta)$, then 
	\begin{linenomath*}
		$$
		\omega(X,Y)=\xi \cdot y - \eta \cdot x.
		$$
	\end{linenomath*}
	We can identify $\omega$ with the isomorphism of $\R^{2n}$ to $\R^{2n}$ such that $\omega^*=-\omega$, with the formula $\omega(X,Y)= \langle \omega X,Y\rangle$. Consider a Riemannian metric $g_X$ on $\R^{2n}$ (which is a measurable function of $X$) to which we associate the dual metric $g_X^\omega$ by
	\begin{linenomath*}
		\[
		\forall T \in \R^{2n}, \quad g_X^\omega(T)= \sup_{0 \neq T' \in \R^{2n}} \frac{\langle \omega T,T'\rangle^2}{g_X(T')}.
		\]
	\end{linenomath*}
	Considering $g_X$ as a matrix associated to positive definite quadratic form on $\R^{2n}$, $g_X^\omega=\omega^*g_X^{-1}\omega$.
	We define the Planck function \cite{nicRodi}, that plays a crucial role in the development of pseudodifferential calculus to be
	\begin{linenomath*}
		\[
		h_g(x,\xi) := \inf_{0\neq T \in \R^{2n}} \Bigg(\frac{g_X(T)}{g_X^\omega(T)}\Bigg)^{1/2}.
		\]	
	\end{linenomath*}
	The uncertainty principle is quantified as the upper bound $h_g(x,\xi)\leq 1$. In the following, we often make use of the strong uncertainty principle, that is, for some $\kappa>0$, we have
	\begin{linenomath*}
		\[
		h_g(x,\xi) \leq (1+|x|+|\xi|)^{-\kappa}, \quad (x,\xi)\in \R^{2n}.
		\] 
	\end{linenomath*}
	Basically, a pseudodifferential calculus is the datum of the metric satisfying some local and global conditions. In our case, it amounts to the conditions on $\P$. The symplectic structure and the uncertainty principle also play a natural role in the constraints imposed on $\Phi$. So we consider $\P=\Phi(|x|)$ to be a monotone increasing function of $|x|$ satisfying following conditions: 
	\begin{linenomath*}
		\begin{alignat}{3}
			\label{sl} 
			1 \; \leq & \quad \Phi(x) &&\lesssim  1+|x| && \quad \text{(sub-linear)} \\
			\label{sv}
			\vert x-y \vert \; \leq & \quad r\Phi(y) && \implies C^{-1}\Phi(y)\leq \Phi(x) \leq C \Phi(y)  && \quad \text{(slowly varying)} \\
			\label{tp}
			&\Phi(x+y) && \lesssim  \Phi(x)(1+|y|)^s && \quad \text{(temperate)}	   		 
		\end{alignat}
	\end{linenomath*}  
	for all $x,y\in\R^n$ and for some $r,s,C>0$. Note that $C \geq 1$ in the slowly varying condition with $x=y$. 
	
	For the sake of calculations arising in the development of symbol calculus related to metrics $g_{\Phi,k}$ and $\tilde{g}_{\Phi,k}$, we need to impose following additional conditions:
	\begin{linenomath*}
		\begin{alignat}{3}  	
			\label{sa}
			|\Phi(x) - \Phi(y)| \leq & \Phi(x+y) && \leq \Phi(x) + \Phi(y)  && \quad  (\text{Subadditive})\\  	
			\label{Phi}
			& |\partial_x^\beta \Phi(x)| && \lesssim \Phi(x) \japx^{-|\beta|}, \\
			&\Phi(ax) &&\leq a\P, \text{ if } a>1\\
			\label{scale}
			& a\P &&\leq \Phi(ax), \text{ if } a \in [0,1]
		\end{alignat}   
	\end{linenomath*}	
	where $\beta \in \mathbb{Z}_+^n$. It can be observed that the above conditions are quite natural in the context of symbol classes. The set of all $\Phi$ satisfying the conditions (\ref{sl})-(\ref{scale}) has a lattice structure. This enables one to induce a corresponding lattice structure on the set of all metrics $g_{\Phi,k}$ of the form given in (\ref{om}). We refer the reader to Appendix I for the details.
	
	In our analysis, we will be using two metrics on the phase space. The first one, $g_{\Phi,k}$, is as in $(\ref{om})$. This metric governs the growth rate (in $x$ and $\xi$) of the symbol of the operator in (\ref{eq1}). For this metric, the Planck function is $h(x,\xi)=h_g(x,\xi)=\Phi(x)^{-1}\japxin^{-1} \leq 1$. The second one, $\tilde{g}_{\Phi,k}$,  is as in (\ref{om2}). This metric governs the growth rate of the lower order terms arising after conjugation with infinite order operator, $e^{\Lambda(t)(\P\la D \ran)^{1/\sigma}}$. Here, the Planck function is $\tilde{h}(x,\xi)= h_{\tilde{g}}(x,\xi)=(\P\japxin)^{\frac{1}{\sigma}-\gamma}$. The metric $\tilde{g}$ satisfies the strong uncertainty principle when $\frac{1}{\sigma}-\gamma < 0$ or  $2 < \sigma $.  
	
	We now demonstrate a procedure for an optimal choice for the function $\Phi(x)$ in the metric (\ref{om})  using the second order one-dimensional model problem (\ref{model}) (discussed in Section \ref{intro}) admitting linear growth in $x \in \R^n$, $|x| \to \infty$ and the singular behavior in $t$. Let us consider the function $A(t,x,\xi)=a(t,x)\xi^2$.   
	The strict hyperbolicity condition on the model problem ensures the existence of $\Phi(x)$. We consider $\Phi(x)$ in such a way that
	\begin{linenomath*}
		\begin{equation}\label{ineq2}
			C^{-1}\P^2 \leq \frac{1}{2} \partial_\xi^2 A(t,x,\xi) \leq C \P^2,  
		\end{equation}
	\end{linenomath*}
	\begin{linenomath*}
		\[
		\text{and} \quad |D_x^\beta D_\xi^\alpha A(t,x,\xi)| \leq C_{\alpha\beta} \P^{2-\beta} \japxi^{2-\alpha},\quad \alpha,\beta \in \mathbb{Z}_+
		\]
	\end{linenomath*}
	for some positive constants $C_{\alpha\beta},C>0$. Since $0< \delta_0 \leq a(t,x)$ for some $\delta_0>0$, we have $1 \leq a(t,x)\delta_0^{-1}$. If $\sqrt{a(t,x)\delta_0^{-1}}$ satisfies the properties (\ref{sl})-(\ref{scale}) for each $t\in [0,T]$, then  we take 
	\[
	\P = \inf_{t\in [0,T]} \sqrt{a(t,x)\delta_0^{-1}}.
	\]
	Due to the lattice structure on the class of functions satisfying (\ref{sl})-(\ref{scale}) (as discussed in Appendix I) such a infimum exists within the class. If $\sqrt{a(t,x)\delta_0^{-1}}$ does not satisfy the required properties, we simply take infimum of all $\Phi$ which satisfy (\ref{ineq2}).  
	
	For higher order equations, we consider $\Phi(x)$ such that
	\begin{linenomath*}
		\[
		C^{-1 }\P \leq \sum_{j=0}^{m-1} \sum_{|\alpha|=m-j}\frac{1}{\alpha!}|\partial_\xi^\alpha A_{m-j}(t,x,\xi)| \leq C\P, \quad \text{and}
		\]
	\end{linenomath*}
	\begin{linenomath*}
		\[
		|D_x^\beta D_\xi^\alpha A_{m-j}(t,x,\xi)| \leq C_{\alpha\beta} \P^{m-j-|\beta|} \japxin^{m-j-|\alpha|},\quad \alpha,\beta \in \mathbb{Z}_+^n.
		\]
	\end{linenomath*}
	Further, we need $\Phi$ to govern the growth of coefficients of lower order terms as well i.e.,
	\begin{linenomath*}
		\[
		|D_x^\beta D_\xi^\alpha B_{m-j}(t,x,\xi)| \leq C_{\alpha\beta} \P^{m-j-1-|\beta|} \japxin^{m-j-1-|\alpha|},\quad \alpha,\beta \in \mathbb{Z}_+^n.
		\]
	\end{linenomath*}
	
	Once an optimal choice for $\Phi$ is made we use it to define the global symbol classes needed for our study.

	\subsection{Global Symbol Classes}\label{Symbols}
	
	We now define the global symbol classes associated to the metric $g_{\Phi,k}$ where
	\begin{linenomath*}
		\[
		g_{\Phi,k}= \frac{|dx|^2}{\P^2} + \frac{|d\xi|^2}{\japxin^2}, \quad \text { for } (x,\xi) \in \R^{2n},
		\]
	\end{linenomath*}
	see \cite[Chapters 1 \& 6]{nicRodi}. Let $m=(m_1,m_2)\in \R^2$. 
	
	\begin{defn}
		$G^{m_1,m_2}_{\Phi,k}$ is the space of all functions $p=p(x,\xi) \in C^\infty(\R^{2n})$ satisfying 
		\begin{linenomath*}
			\begin{equation}
				\label{sym1}
				|\partial_\xi^\alpha  D_x^\beta p(x,\xi)| < C_{\alpha \beta} \japxin^{m_1-|\alpha|} \P^{m_2-|\beta|},
			\end{equation}
		\end{linenomath*}	
		for $C_{\alpha \beta}>0$ and for all multi-indices $\alpha$ and $\beta$.	
	\end{defn}
	We need the following symbol classes with Gevrey regularity that will be helpful in the proof of Theorem \ref{result2}.
	Let $\mu$, $\nu$ be real numbers with $\mu \geq 1$, $\nu \geq 1$.
	\begin{defn}
		For every $C>0$, we denote by $AG_{\Phi,k;\mu,\nu}^{m_1,m_2}(C)$ the Banach space of all symbols $p(x,\xi) \in G^{m_1,m_2}_{\Phi,k}$ such that the constant $C_{\alpha \beta}>0$ in (\ref{sym1}) is of the form, 
		\begin{linenomath*}
			\[
			C_{\alpha\beta}= B \ C^{|\alpha|+|\beta|}(\alpha!)^{\mu} (\beta!)^{\nu},
			\]
		\end{linenomath*}
		for some $B>0$ independent of $\alpha$ and $\beta$. 
	\end{defn}
	This space is endowed with the norm defined by the optimal quantity $B$ in the above equation.
	We set 
	\begin{linenomath*}
		\[
		AG_{\Phi,k;\mu,\nu}^{m_1,m_2}=\lim_{\substack{\longrightarrow \\ C \to +\infty}} AG_{\Phi,k;\mu,\nu}^{m_1,m_2}(C)
		\]
	\end{linenomath*}
	with the topology of inductive limit of an increasing sequence of Banach spaces. 
	
	After the conjugation by infinte order pseudodifferential operator, $e^{\Lambda(t)(\P\la D\ran)^{1/\sigma}}$, the growth estimates for the lower order terms are governed by the metric $\tilde{g}_{\Phi,k}$ given in (\ref{om2}). We will now define the symbol classes associated with this metric.
	
	\begin{defn}
		For every $\sigma\geq 3$, we denote by $G_{\Phi,k,\sigma}^{m_1,m_2}$
		the space of all functions $p=p(x,\xi) \in C^\infty(\R^{2n})$ satisfying 
		\begin{linenomath*}
			\begin{equation}
				\label{sym3}
				|\partial_\xi^\alpha  D_x^\beta p(x,\xi)| < C_{\alpha \beta} \japxin^{m_1-\gamma|\alpha|+|\beta|/\sigma} \Phi(x)^{m_2-\gamma|\beta|+|\alpha|/\sigma},
			\end{equation}	
		\end{linenomath*}
		for $C_{\alpha \beta}>0$, $\gamma = 1 - \frac{1}{\sigma}$ and for all multi-indices $\alpha$ and $\beta$.	
	\end{defn}	
	Morever, we shall need the following Gevrey variant
	of the above symbol class.
	
	\begin{defn}
		We denote by $AG_{\Phi,k,\sigma;\mu,\nu}^{m_1,m_2}$ the Banach space of all symbols $p(x,\xi) \in G_{\Phi,\sigma}^{m_1,m_2}$ such that the constant $C_{\alpha \beta}>0$ in (\ref{sym3}) is of the form, 
		\begin{linenomath*}
			\[
			C_{\alpha,\beta}= B \  C^{|\alpha|+|\beta|}(\alpha!)^{\mu} (\beta!)^{\nu},
			\]
		\end{linenomath*}
		for some $C>0$ and $B>0$ independent of $\alpha$ and $\beta$. 
	\end{defn}
	
	Inspired from \cite{AscaCappi2}, we introduce the following symbol class in order to deal with the symbols which are polynomial in $\xi$ and Gevrey of order $\sigma \geq 3$ with respect to $x$. Here we impose analytic estimates with respect to $\xi$ on an exterior domain of $\R^{2n}$. A suitable class for our purpose is defined as follows.
	\begin{defn}
		We shall denote by $AG^{m_1,m_2}_{\Phi,k;\sigma}$ the space of all symbols $p(x,\xi) \in AG_{\Phi,\sigma;\sigma,\sigma}^{m_1,m_2}$ satisfying the following condition: there exist positive constants $B$, $C$ such that
		\begin{linenomath*}
			\[
			\begin{aligned}
				\sup_{\alpha,\beta \in \mathbb{N}^n} \sup_{\P\japxin\geq B|\alpha|^\sigma} C^{-|\alpha|-|\beta|} &(\alpha!)^{-1} (\beta!)^{-\sigma} \japxin^{-m_1+|\alpha|} \\ & \times\Phi(x)^{-m_2+|\beta|} |D_\xi^\alpha \partial_x^\beta p(x,\xi)| < + \infty.
			\end{aligned}			
			\]	
		\end{linenomath*}
	\end{defn}	
	The following inclusions obviously hold:
	\begin{linenomath*}
		\[
		AG_{\Phi,k;1,\sigma}^{m_1,m_2} \subset AG^{m_1,m_2}_{\Phi,k;\sigma} \subset AG_{\Phi,k;\sigma,\sigma}^{m_1,m_2} \subset AG_{\Phi,k,\sigma;\sigma,\sigma}^{m_1,m_2}.
		\]
	\end{linenomath*}
	Given a symbol $p \in AG^{m_1,m_2}_{\Phi,k,\sigma;\sigma,\sigma}$, we can consider the associated pseudodifferential operator $P=p(x,D_x)$ defined by the following oscillatory integral
	\begin{linenomath*}
		\begin{align}
			Pu(x)& =\iint\limits_{\R^{2n}}e^{i(x-y)\cdot\xi}p(x,\xi){u}(y)dy\textit{\dj}\xi \nonumber\\
			\label{pdo}
			& = \int\limits_{\R^n}e^{ix\cdot\xi}p(x,\xi)\hat{u}(\xi) \textit{\dj}\xi,
		\end{align}
	\end{linenomath*}
	where $u \in \mathcal{S}(\R^n)$ and $\textit{\dj}\xi = (2 \pi)^{-n}d\xi$. We shall denote by $OPAG^{m_1,m_2}_{\Phi,k,\sigma;\sigma,\sigma}$, the space of all operators of the form (\ref{pdo}) defined by a symbol $p \in AG^{m_1,m_2}_{\Phi,k,\sigma;\sigma,\sigma}$. Operators from $OPAG^{m_1,m_2}_{\Phi,k,\sigma;\sigma,\sigma}$ map continuously $\mathcal{S}^\theta_\theta(\R^n)$ into $\mathcal{S}^\theta_\theta(\R^n)$, ${\mathcal{S}^\theta_\theta}'(\R^n)$ into ${\mathcal{S}^\theta_\theta}'(\R^n)$, for $\theta \geq \sigma$, (see Appendix II of this work and \cite[Section 6.3]{nicRodi}).
	
	As far as the calculi of the operators in $OPAG^{m_1,m_2}_{\Phi,k;\sigma}$ and $OPAG^{m_1,m_2}_{\Phi,k,\sigma;\sigma,\sigma}$ are concerned one can easily construct them by using the properties of $\P$, see Appendix II and III below for details and see also \cite{nicRodi}.  We follow the similar standard arguments given in \cite[Section 6.3]{nicRodi} and \cite[Appendix A]{AscaCappi2} for the calculus of operators in $OPAG^{m_1,m_2}_{\Phi,k,\sigma;\sigma,\sigma}$ and $OPAG^{m_1,m_2}_{\Phi,k;\sigma}$, respectively.

			\section{Conjugation by an Infinite Order Operator}\label{Conj}
	In the global setting, solutions experience both loss of derivatives and loss of decay. In order to overcome the difficulty of tracking a precise loss in our context we introduce a class of parameter dependent infinite order pseudodifferential operators of the form $e^{\Lambda(t) (\P\la D \ran)^{1/\sigma}}$, defined using the Planck function $\h$ and a continuous function $\Lambda(t)$. These operators compensate, microlocally, the loss of regularity of solutions. The operators with loss of regularity (both derivatives and decay) are transformed to \enquote{good} operators by conjugation with such infinite order operator. This helps us to derive \enquote{good} a priori estimates of solutions in the Sobolev space associated with the operator $e^{\Lambda(t) (\P\la D \ran)^{1/\sigma}}$. 
	
	In this section, we perform a conjugation of operators from $OPAG^{m_1,m_2}_{\Phi,k;\sigma,\sigma}$ by $\exp\{\Lambda(t) (\P\la D \ran)^{1/\sigma}\},\sigma \geq 3$. Here we assume that $\Lambda(t)$ is a  continuous function for $t \in [0,T]$. The following proposition gives  an upper bound on the function $\Lambda(t)$ for the conjugation to be well defined.
	\begin{thm}\label{conju}
		Let $p \in AG^{m_1,m_2}_{\Phi,k;\sigma,\sigma}(C)$ for some $\sigma \geq 3$, $C>0$, $m = (m_1,m_2) \in \R^2$ and $\Lambda = \Lambda(t)$ be a continuous function of $t \in[0,T]$. Then, there exists $\Lambda_0>0$ depending only on $C$ such that for $\Lambda(t) > 0$ with  $\Lambda(t)<\Lambda_0$,
		\begin{linenomath*}
			\begin{align}
				\exp{(\Lambda(t) (\P\la D \ran)^{1/\sigma})}&p(x,D)\exp{(-\Lambda(t) (\P\la D \ran)^{1/\sigma})} \nonumber\\
				\label{remconj}
				& = p(x,D)+ \sum_{k=1}^{3} r_{\Lambda}^{(k)}(t,x,D_x),
			\end{align}
		\end{linenomath*}
		where $r_{\Lambda}^{(k)}(t,x,D_x)$ for $k=1,2,3$ are operators with the symbols in\\ $C([0,T];AG^{-\infty,m_2-\gamma}_{\Phi,k,\sigma;\sigma,\sigma})$, $C([0,T];AG^{m_1-\gamma,-\infty}_{\Phi,k,\sigma;\sigma,\sigma})$ and  $C([0,T];AG^{-\infty,-\infty}_{\Phi,k,\sigma;\sigma,\sigma})$, respectively, for $\gamma=1 - \frac{1}{\sigma}$.
	\end{thm}  
	
	Before we prove the theorem, let us see its implications to partial differential operator $p(x,D) \in OPAG^{m_1,m_2}_{\Phi,k;\sigma,\sigma} $ of the form, 
	\begin{linenomath*}
		\[
		p(x,D)= \sum\limits_{\substack{|\alpha|\leq m_1 \\ |\beta| \leq m_2}} C_{\alpha\beta} a_\beta(x)D^\alpha
		\]
	\end{linenomath*}
	where $(m_1,m_2) \in \mathbb{Z}_+^2$ and for each $j \leq m_2$, $\sum\limits_{|\beta|=j}|a_\beta(x)| \sim \P^{|\beta|}$. Let us decompose $p(x,\xi)$ as:
	\begin{linenomath*}
		\[
		p(x,\xi) = \sigma^{(m_1)}(p)(x,\xi) + \sigma_{(m_2)}(p)(x,\xi) - \sigma^{(m_1)}_{(m_2)}(p)(x,\xi) + p_l(x,\xi)
		\]
	\end{linenomath*}
	where
	\begin{linenomath*}
		\[
		\sigma^{(m_1)}_{(m_2)}(p)(x,\xi)=  \sum\limits_{\substack{|\alpha|= m_1 \\ |\beta| = m_2}} C_{\alpha\beta} a_\beta(x)\xi^\alpha, \quad \sigma^{(m_1)}(p)(x,\xi)=  \sum\limits_{\substack{|\alpha|= m_1 \\ |\beta| \leq m_2}} C_{\alpha\beta} a_\beta(x)\xi^\alpha, \quad
		\]
	\end{linenomath*}
	\begin{linenomath*}
		\[
		\sigma_{(m_2)}(p)(x,\xi)=  \sum\limits_{\substack{|\alpha| \leq m_1 \\ |\beta| = m_2}} C_{\alpha\beta} a_\beta(x)\xi^\alpha, \quad p_{l}(x,\xi)=  \sum\limits_{\substack{|\alpha|< m_1 \\ |\beta| < m_2}} C_{\alpha\beta} a_\beta(x)\xi^\alpha.
		\]
	\end{linenomath*}
	Here $p_l(x,\xi)$ contains the lower order terms with respect to both $x$ and $\xi$. Note that $\sigma^{(m_1)}(p), \sigma_{(m_2)}(p)$ and $\sigma^{(m_1)}_{(m_2)}(p)$ are \enquote{similar} to interior, exit and bi-homogeneous principal parts of a classical symbol, respectively, see \cite[Section 3.2]{nicRodi}.
	From (\ref{remconj}), it is clear that the metric governing the growth estimates of the lower order terms changes
	after the conjugation by combining the lower order terms with the remainder. That is, we have 
	\begin{linenomath*} 
		\begin{align*}
			&\exp{(\Lambda(t) (\P\la D \ran)^{1/\sigma})}  p(x,D)\exp{(-\Lambda(t) (\P\la D \ran)^{1/\sigma})} \\
			& \quad =\sigma^{(m_1)}(p)(x,D) + \sigma_{(m_2)}(p)(x,D) - \sigma^{(m_1)}_{(m_2)}(p)(x,D)+ \sum_{k=1}^{3} p_{l,\Lambda}^{(k)}(t,x,D_x)
		\end{align*}
	\end{linenomath*}
	where $p_{l,\Lambda}^{(k)}(t,x,D_x)$ are operators with the symbols in $C([0,T];AG^{m_1-1,m_2-\gamma}_{\Phi,k,\sigma;\sigma,\sigma})$, $C([0,T];AG^{m_1-\gamma,m_2-1}_{\Phi,k,\sigma;\sigma,\sigma})$ and  $C([0,T];AG^{-\infty,-\infty}_{\Phi,k,\sigma;\sigma,\sigma})$, for $k=1,2,3,$ respectively.
	
	To prove the Theorem \ref{conju}, we need the following lemma, which can be given an inductive proof.
	\begin{lem}
		Let $\varepsilon \neq 0$, $\sigma \geq 3$. Then, for every $\alpha,\beta \in \mathbb{Z}^n_+$, we have
		\begin{linenomath*}
			\[
			\partial_x^\beta\partial_\xi^\alpha e^{\varepsilon (\P\japxin)^{1/\sigma}} \leq (C\varepsilon)^{|\alpha|+|\beta|}  \alpha! \beta! \P^{-\gamma|\beta|+|\alpha|/\sigma} \japxin^{-\gamma|\alpha|+|\beta|/\sigma} e^{\varepsilon (\P\japxin)^{1/\sigma}}.
			\]
		\end{linenomath*}
	\end{lem}
	
	\begin{proof}[\bfseries{Proof of Theorem~\ref{conju}}]
		Throughout this proof we write $\Lambda$ in place of $\Lambda(t)$ and denote $\P\japxin$ by $\psi(x,\xi)$ for the sake of simplicity of notation. 
		Let $p_{\Lambda,\sigma}(x,\xi)$ be the symbol of the operator
		\begin{linenomath*}
			\[
			\exp{(\Lambda (\P\la D \ran)^{1/\sigma})}p(x,D)\exp{(-\Lambda (\P\la D \ran)^{1/\sigma})}.
			\] 
		\end{linenomath*}
		Then  $p_{\Lambda,\sigma}(x,\xi)$ can be written in the form of an oscillatory integral as follows:
		\begin{linenomath*}
			\begin{align}\label{conj}
				p_{\Lambda,\sigma}(x,\xi)= \int \dots \int e^{-iy\cdot\eta}e^{-iz\cdot\zeta}&e^{\Lambda \psi(x,\xi+\zeta+\eta)^{1/\sigma}} p(x+z,\xi+\eta) \\
				& \times e^{-\Lambda \psi(x+y,\xi)^{1/\sigma}}dz\textit{\dj}\zeta dy\textit{\dj}\eta, \nonumber
			\end{align}
		\end{linenomath*}
		Taylor expansions of $\exp\{\Lambda \psi(x,\xi)^{1/\sigma}\}$ in the first and second variables, respectively, are
		\begin{linenomath*} 
			\begin{align*}
				e^{-\Lambda \psi(x+y,\xi)^{1/\sigma}} &= e^{-\Lambda \psi(x,\xi)^{1/\sigma}} + \sum_{j=1}^{n} \int_{0}^{1}y_j \partial_{w_j'}e^{-\Lambda \psi(w',\xi)^{1/\sigma}} \Big|_{w'=x+\theta_1y}d\theta_1, \text{ and }\\
				e^{\Lambda \psi(x,\xi+\zeta+\eta)^{1/\sigma}} &= e^{\Lambda \psi(x,\xi)^{1/\sigma}} + \sum_{i=1}^{n} \int_{0}^{1}(\zeta_i + \eta_i) \partial_{w_i}e^{\Lambda \psi(x,w)^{1/\sigma}} \Big|_{w=\xi+\theta_2(\eta+\zeta)}d\theta_2.		
			\end{align*}
		\end{linenomath*}
		We can write $p_{\Lambda,\sigma}$ as 
		\begin{linenomath*}
			\[
			p_{\Lambda,\sigma}(x,\xi)=p(x,\xi) + \sum_{l=1}^{3} r_{\Lambda}^{(l)}(t,x,\xi)\quad \text{ where }
			\]	
		\end{linenomath*}
		\begin{linenomath*}	 
			\[
			r_{\Lambda}^{(l)}(x,\xi) = \int \cdots \int e^{-iy\cdot\eta}e^{-iz\cdot\zeta} I_l p(x+z,\xi+\eta) dz\textit{\dj}\zeta dy\textit{\dj}\eta,
			\]	
		\end{linenomath*}	
		and $I_l$, $l=1,2,3$ are as follows:
		\begin{linenomath*}
			\begin{align*}
				I_1 &= e^{\Lambda \psi(x,\xi)^{1/\sigma}} \sum_{j=1}^{n} \int_{0}^{1}y_j \partial_{w_j'}e^{-\Lambda \psi(w',\xi)^{1/\sigma}} \Big|_{w'=x+\theta_1y}d\theta_1,\\
				I_2 &=  e^{-\Lambda \psi(x,\xi)^{1/\sigma}}  \sum_{i=1}^{n} \int_{0}^{1}(\zeta_i+ \eta_i) \partial_{w_i}e^{\Lambda \psi(x,w)^{1/\sigma}} \Big|_{w=\xi+\theta_2(\zeta + \eta)}d\theta_2 \quad \text{ and }\\
				I_3 &= 
				\Bigg(\sum_{i=1}^{n} \int_{0}^{1}(\zeta_i+\eta_i) \partial_{w_i}e^{\Lambda \psi(x,w)^{1/\sigma}} \Big|_{w=\xi+\theta_2(\zeta+\eta)}d\theta_2\Bigg) \\
				&\quad \times\Bigg(\sum_{j=1}^{n} \int_{0}^{1}y_j \partial_{w'_j}e^{-\Lambda \psi(w',\xi)^{1/\sigma}}\Big|_{w'=x+\theta_1y}d\theta_1\Bigg).
			\end{align*}
		\end{linenomath*}		
		We will now determine the growth estimate for $r^{(1)}_\Lambda(x,\xi)$ using integration by parts. For $\alpha, \beta,\kappa \in \mathbb{Z}_+^n$ and $l \in \mathbb{Z}_+$ we have	
		\begin{linenomath*}
			\[
			\begin{aligned}
				&\partial_\xi^\alpha\partial_x^\beta r^{(1)}_\Lambda(t,x,\xi)\\
				&= \sum_{j=1}^{n} \sum_{\tiny{\beta'+ \beta'' \leq \beta}}\sum_{\scriptsize{\alpha'+\alpha'' \leq \alpha}} \int \dots \int y^{-\kappa} y^{\kappa} e^{-iy\cdot\eta}e^{-iz\cdot\zeta} 
				(\partial_\xi^{\alpha'} \partial_x^{\beta'} D_{\xi_j}p)(x+z,\xi+\eta)\\	
				&\quad\times \int_{0}^{1} \partial_\xi^{\alpha''}\partial_x^{\beta''}\partial_{w'_j}	e^{\Lambda (\P^{1/\sigma}- \Phi(w')^{1/\sigma})\japxin^{1/\sigma}} 	\Big|_{w'=x+\theta_1y} d\theta_1 dz\textit{\dj}\zeta dy\textit{\dj}\eta.
			\end{aligned}
			\]
		\end{linenomath*}		
		\begin{linenomath*}
			\[
			\begin{aligned}
				&\partial_\xi^\alpha\partial_x^\beta r^{(1)}_\Lambda(t,x,\xi)\\
				&= \sum_{j=1}^{n} \sum_{\tiny{\beta'+ \beta'' \leq \beta}}\sum_{\scriptsize{\alpha'+\alpha'' \leq \alpha}} \int \dots \int y^{-\kappa} e^{-iy\cdot\eta}e^{-iz\cdot\zeta} \la \eta \ra_N ^{-2l} \la z \ra_N ^{-2l}\la D_\zeta\ra_N^{2l} \la \zeta \ra_N^{-2l}  \\
				&\quad\times \la D_z\ra_N^{2l} D_\eta^\kappa (\partial_\xi^{\alpha'} \partial_x^{\beta'} D_{\xi_j}p)(x+z,\xi+\eta)\\	
				&\quad\times \int_{0}^{1} \la D_y \ra_N^{2l} \partial_\xi^{\alpha''}\partial_x^{\beta''}\partial_{w'_j}	e^{\Lambda (\P^{1/\sigma}- \Phi(w')^{1/\sigma})\japxin^{1/\sigma}} 	\Big|_{w'=x+\theta_1y} d\theta_1 dz\textit{\dj}\zeta dy\textit{\dj}\eta.
			\end{aligned}
			\]
		\end{linenomath*}
		Let $E_1(t,x,y,\xi) = \exp\{\Lambda (\Phi(x)^{1/\sigma}- \Phi(x+\theta_1y)^{1/\sigma}) \japxin^{1/\sigma}\}$.
		Note that for $|y| \geq 1$ we have $\la y \ra \leq \sqrt{2}|y|$ and in the case $|y| <1$ we have $\la y\ra < \sqrt{2}$. Using these estimates along with the fact that $\la y \ra^{-|\kappa|} \leq \Phi(y)^{-|\kappa|}$ we have
		\begin{linenomath*}
			\[
			\begin{aligned}
				|\partial_\xi^\alpha\partial_x^\beta r^{(1)}_\Lambda(t,x,\xi)|&\leq  C_1^{|\alpha|+|\beta|+2} \P^{m_2-\gamma-\gamma|\beta|+|\alpha|/\sigma} \japxin^{m_1-\gamma-\gamma|\alpha|+|\beta|/\sigma}    \\
				& \quad \times \sum_{\tiny{\beta'+ \beta'' \leq \beta}}\sum_{\scriptsize{\alpha'+\alpha'' \leq \alpha}}  \int \dots \int \Phi(z)^{|m_2-|\beta'||} \Phi(y)^{\gamma|\beta''|+|\alpha''|/\sigma}  \\ 
				& \quad \times \la \eta \ran^{|m_1-1-|\alpha'||+\gamma|\alpha''|+|\beta''|/\sigma + |\kappa| -2l}  |\kappa|!^\sigma\bigg(\frac{C_1 2^\sigma}{\Phi(y)\japxin}\bigg)^{|\kappa|} \\
				& \quad \times \japxin^{2l/\sigma} E_1(t,x,y,\xi)\la z \ran^{-2l} \la \zeta \ran^{-2l}dz\textit{\dj}\zeta dy\textit{\dj}\eta.
			\end{aligned}
			\]
		\end{linenomath*}
		Given $\alpha, \beta$ and $\kappa$, we choose $l$ such that $2l > n + \max\{m_1,m_2\}+|\alpha|+|\beta|+ |\kappa|$. So that
		\begin{linenomath*}
			\[
			\begin{aligned}
				|\partial_\xi^\alpha\partial_x^\beta r^{(1)}_\Lambda(t,x,\xi)| & \leq  C_1^{|\alpha|+|\beta|+2}\japxin^{m_1-\gamma-\gamma|\alpha|+(|\beta|+2l)/\sigma} \P^{m_2-\gamma-\gamma|\beta|+|\alpha|/\sigma}\\
				& \quad \times \int  |\kappa|!^\sigma \bigg(\frac{C_1 2^\sigma}{\Phi(y)\japxin}\bigg)^{|\kappa|} \Phi(y)^{2l} E_1(x,y,\xi) dy.
			\end{aligned}
			\]
		\end{linenomath*}
		Noting the inequality (see \cite[Lemma 6.3.10]{nicRodi})
		\begin{linenomath*}
			\[
			\inf_{j \in \mathbb{Z}_+} j!^\sigma \bigg(\frac{C_1 2^\sigma}{\Phi(y) \japxin}\bigg)^{j} \leq C' e^{-c_1(\Phi(y)\japxin)^{1/\sigma}},
			\]
		\end{linenomath*}
		for some positive constants $C'$ and $c_1$ where $C'$ depends only on $C_1$ and $c_1$ on $n$ and $C_1$, we have 
		\begin{linenomath*}
			\[
			\begin{aligned}
				|\partial_\xi^\alpha\partial_x^\beta r^{(1)}_\Lambda(t,x,\xi)| \: \leq & \quad  C^{|\alpha|+|\beta|+2}\japxin^{m_1-\gamma-\gamma|\alpha|+|\beta|/\sigma} \P^{m_2-\gamma-\gamma|\beta|+|\alpha|/\sigma}\\
				&\times \int 	e^{-c_1\Phi(y)^{1/\sigma}\japxin^{1/\sigma}} \Phi(y)^{2l} E_1(x,y,\xi) dy.		
			\end{aligned}
			\]
		\end{linenomath*}
		Let $l' \in \mathbb{Z}^+$ such that $\frac{l'}{\sigma} \geq  l$. Then, we have $e^{-c_1 \Phi(y)^{1/\sigma}\japxin ^{1/\sigma}} \Phi(y)^{2l'/\sigma} \leq (2l')!\;e^{-\frac{c_1}{2} \Phi(y)^{1/\sigma}\japxin ^{1/\sigma}}$. Hence,
		\begin{linenomath*}
			\[
			\begin{aligned}
				|\partial_\xi^\alpha\partial_x^\beta r^{(1)}_\Lambda&(t,x,\xi)| \\
				&\leq   C_1^{|\alpha|+|\beta|+2}\japxin^{m_1-\gamma-\gamma|\alpha|+|\beta|/\sigma} \P^{m_2-\gamma-\gamma|\beta|+|\alpha|/\sigma}\\
				& \quad \times \int \exp\Big\{(\Lambda\P^{1/\sigma}- \Lambda \Phi( x+\theta_1y)^{1/\sigma} - \frac{c_1}{2} \Phi(y)^{1/\sigma})\japxin^{1/\sigma}\Big\}
				dy.
			\end{aligned}
			\]
		\end{linenomath*}
		For $|x| \leq |y|$, clearly $\Phi(x)^{1/\sigma} - \Phi(x+\theta_1y)^{1/\sigma} \leq \Phi(y)^{1/\sigma} $. For $|x| \geq |y|$, we have 
		\begin{linenomath*}
			\begin{equation}\label{Pineq}
				\begin{aligned}
					\Phi(x)^{1/\sigma} - \Phi(x+\theta_1y)^{1/\sigma} &\leq \Phi(x)^{1/\sigma} - (\P-\Phi(\theta_1y))^{1/\sigma} \\
					&\leq \Phi(x)^{1/\sigma} - (\P^{1/\sigma}-\Phi(\theta_1 y)^{1/\sigma})\leq\Phi(y)^{1/\sigma}.
				\end{aligned}  			
			\end{equation}
		\end{linenomath*} 		
		Since $c_1$ is independent of $\Lambda$, there exists $\Lambda^{(1)}>0$ (in fact, $\Lambda^{(1)}=c_1/2$) such that, for $\Lambda = \Lambda(t) <\Lambda^{(1)}$ we obtain the estimate
		\begin{linenomath*}
			\[
			|\partial_\xi^\alpha\partial_x^\beta r^{(1)}_\Lambda(t,x,\xi)|\leq C^{|\alpha|+|\beta|+2} \P^{m_2-\gamma-\gamma|\beta|+|\alpha|/\sigma} e^{-\frac{c_1}{8}\japxin^{1/\sigma}}
			\] 	
		\end{linenomath*}	
		Thus $r^{(1)}_\Lambda \in C([0,T];AG^{-\infty,m_2-\gamma}_{\Phi,\sigma;\sigma,\sigma})$.
		
		In a similar fashion, we will determine the growth estimate for $r^{(2)}_\Lambda(t,x,\xi)$. Let $\alpha, \beta,\kappa \in \mathbb{Z}_+^n$ and $l \in \mathbb{Z}_+$. Then		
		\begin{linenomath*}
			\[
			\begin{aligned}
				&\partial_\xi^\alpha\partial_x^\beta r^{(2)}_\Lambda(t,x,\xi)\\
				&= \sum_{i=1}^{n} \sum_{\tiny{\beta'+ \beta'' \leq \beta}}\sum_{\scriptsize{\alpha'+\alpha'' \leq \alpha}} \int \dots \int \eta^{-\kappa} \eta^\kappa e^{-iy\cdot\eta} \zeta^{-\kappa} \zeta^\kappa e^{-iz\cdot\zeta} \la z \ran ^{-2l} \la \eta \ran ^{-2l}\la D_y\ran^{2l} \\
				&\quad\times  \la y \ran^{-2l} \la D_\zeta\ran^{2l}   \la D_\eta\ran^{2l} (\partial_\xi^{\alpha'} \partial_x^{\beta'} D_{x_i}p)(x+z,\xi+\eta)\\
				& \quad \times\int_{0}^{1} \partial_\xi^{\alpha''}\partial_x^{\beta''}\partial_{w_i} e^{\Lambda \P^{1/\sigma}(\la w \ran^{1/\sigma}-\japxin^{1/\sigma})}\Big|_{w=\xi+\theta_2(\eta+\zeta)}	
				d\theta_2 dz\textit{\dj}\zeta dy\textit{\dj}\eta,
			\end{aligned}
			\]	
		\end{linenomath*}
		\begin{linenomath*}
			\[
			\begin{aligned}
				&\partial_\xi^\alpha\partial_x^\beta r^{(2)}_\Lambda(t,x,\xi)\\
				&= \sum_{i=1}^{n} \sum_{\tiny{\beta'+ \beta'' \leq \beta}}\sum_{\scriptsize{\alpha'+\alpha'' \leq \alpha}} \int \dots \int \eta^{-\kappa} e^{-iy\cdot\eta} \zeta^{-\kappa} e^{-iz\cdot\zeta} D_y^\kappa D_z^\kappa \la z \ran ^{-2l} \la \eta \ran ^{-2l} \la D_y\ran^{2l} \\
				&\quad\times \la y \ran^{-2l}   \la D_\zeta\ran^{2l}   \la D_\eta \ran^{l} (\partial_\xi^{\alpha'} \partial_x^{\beta'} D_{x_i}p)(x+z,\xi+\eta)\\
				& \quad \times\int_{0}^{1} \partial_\xi^{\alpha''}\partial_x^{\beta''}\partial_{w_i} e^{\Lambda \P^{1/\sigma}(\la w \ran^{1/\sigma}-\japxin^{1/\sigma})}\Big|_{w=\xi+\theta_2(\eta+\zeta)}	
				d\theta_2 dz\textit{\dj}\zeta dy\textit{\dj}\eta.
			\end{aligned}
			\]	
		\end{linenomath*}
		Let $E_2(t,x,\xi,\eta,\zeta) = \exp\{\Lambda\P^{1/\sigma}(\la \xi+\theta_2(\eta+\zeta) \ran^{1/\sigma}-\japxin^{1/\sigma})\}$.	
		Using the easy to show inequality $\Phi(x+z)^s \leq 2^{|s|} \P^s \Phi(z)^{|s|}, \forall s \in \R$, we have	
		\begin{linenomath*}
			\[
			\begin{aligned}
				|\partial_\xi^\alpha\partial_x^\beta r^{(2)}_\Lambda(t,x,\xi)| & \leq  C_2^{|\alpha|+|\beta|+2}\japxin^{m_1-\gamma-\gamma|\alpha|+|\beta|/\sigma} \P^{m_2-\gamma-\gamma|\beta|+|\alpha|/\sigma}  \\
				& \quad \times \sum_{\tiny{\beta'+ \beta'' \leq \beta}}\sum_{\scriptsize{\alpha'+\alpha'' \leq \alpha}} \int \dots \int
				\Phi(z)^{|m_2-|\beta'||+1} \la \eta \ran^{|m_1-|\alpha'||}\\
				& \quad \times \la \eta \ran^{-2l} \la \eta + \zeta \ran^{\gamma(1+|\alpha''|)+|\beta''|/\sigma} |\kappa|!^\sigma\Big(\frac{C_2 2^\sigma}{\P \la \zeta \ran \la \eta \ran}\Big)^{|\kappa|}\\		
				& \quad \times \la z \ran ^{-2l+|\kappa|} \la y\ran^{-2l-|\kappa|} E_2(t,x,\xi,\eta,\zeta) dz\textit{\dj}\zeta dy\textit{\dj}\eta.
			\end{aligned}
			\]
		\end{linenomath*}
		In this case we choose $l$ such that $2l > 2(n+1) + \max\{m_1,m_2\}+|\alpha|+|\beta|+ |\kappa|$. Noting that $(\la \eta \ran \la \zeta \ran)^{-1} \leq \la \zeta + \eta \ran ^{-1}$ and
		\begin{linenomath*}
			\[
			\inf_{j \in \mathbb{Z}_+} j!^\sigma \bigg(\frac{C_2 2^\sigma}{\P \la \zeta + \eta \ran}\bigg)^{j} \leq C' e^{-c_2(\P \la \zeta + \eta \ran)^{1/\sigma}},
			\]
		\end{linenomath*}
		for some $c_2>0$. Thus we have
		\begin{linenomath*}
			\[
			\begin{aligned}
				|&\partial_\xi^\alpha\partial_x^\beta r^{(2)}_\Lambda(t,x,\xi)| \\
				&\leq  C^{|\alpha|+|\beta|+2}\japxin^{m_1-\gamma-\gamma|\alpha|+|\beta|/\sigma} \P^{m_2-\gamma-\gamma|\beta|+(|\alpha|+l)/\sigma} \int \int
				\la \eta \ran^{-2(n+1)} \\
				& \quad \times  \exp\{\P^{1/\sigma}(\Lambda\la \xi+(\eta+\zeta) \ran^{1/\sigma}-\Lambda\japxin^{1/\sigma}-\frac{c_2}{2}\la \zeta + \eta \ran^{1/\sigma})\} \textit{\dj}\zeta \textit{\dj}\eta.
			\end{aligned}
			\] 
		\end{linenomath*}	
		For $\la \xi+\eta+\zeta\ran \leq 3 \la \eta+\zeta\ran$, we have $\vert \la \xi+\eta+\zeta\ran^{1/\sigma} - \japxin^{1/\sigma} \vert \leq 3 \la \eta+\zeta\ran^{1/\sigma}$. For $\la \xi+\eta+\zeta\ran \geq 3 \la \eta+\zeta\ran$, that is, $\japxin \leq 2 \la \eta+\zeta\ran$, we have
		\begin{linenomath*}
			\begin{equation}\label{xineq}
				\la \xi+\eta+\zeta\ran^{1/\sigma} - \japxin^{1/\sigma} \leq |\eta+\zeta| (\japxin-\la \eta+\zeta\ran)^{\frac{1}{\sigma}-1} \leq \la \eta+\zeta\ran^{1/\sigma}.
			\end{equation}
		\end{linenomath*}		
		Since $c_2$ is independent of $\Lambda$, there exists $\Lambda^{(2)}>0$ (in fact, $\Lambda^{(2)}=c_2/12$) such that, for $\Lambda = \Lambda(t)<\Lambda^{(2)}$ we obtain the estimate
		\begin{linenomath*}
			\[
			|\partial_\xi^\alpha\partial_x^\beta r^{(2)}_\Lambda(t,x,\xi)|\leq C^{|\alpha|+|\beta|+2} \japxin^{m_1-\gamma-\gamma|\beta|+|\alpha|/\sigma} e^{-\frac{c_2}{8}\P^{1/\sigma}}.
			\] 	
		\end{linenomath*}		
		Thus $r^{(2)}_\Lambda \in C([0,T];AG^{m_1-\gamma,-\infty}_{\Phi,k,\sigma;\sigma,\sigma})$.
		By similar techniques used in the case of $ r^{(1)}_\Lambda$ and $ r^{(2)}_\Lambda$, one can show that $r^{(3)}_\Lambda \in C([0,T]; AG^{-\infty,-\infty}_{\Phi,k,\sigma;\sigma,\sigma})$. Taking $\Lambda_0 = \min\{\Lambda^{(1)},\Lambda^{(2)}\}$, proves the theorem.  
	\end{proof}
	\setcounter{rmk}{2}
	\begin{rmk}
		\begin{enumerate}
			\item The conjugation of Theorem \ref{conju} can also be performed by starting with a symbol $p\in AG_{\Phi,k,\sigma;\sigma,\sigma}^{m_1,m_2}$.
			
			\item If $\P \equiv C$ for some $C \geq 1$, then the proof of Theorem \ref{conju} takes simpler form as in \cite[Proposition 2.3]{Kaji}. In such case, for $C_0=C^{1/\sigma}$ we have
			\begin{linenomath*}
				\[
				e^{\Lambda(t)C_0 \la D_x\ra^{1/\sigma}} \: p(x,D_x) \: e^{-\Lambda(t)C_0 \la D_x\ra^{1/\sigma}} = p(x,D_x) + r_\Lambda(t,x,D_x),
				\]
			\end{linenomath*}
			where $r_\Lambda(t,x,\xi)$ is in the H\"{o}rmander class $S^{m_1-\gamma}_{\gamma,0}$, $\gamma=1-\frac{1}{\sigma}$, for each $t$.
		\end{enumerate}		
	\end{rmk}
	Next, we prove two corollaries of Theorem \ref{conju} which will be helpful in making change of variables in the proof of the main result.
	\setcounter{cor}{3}
	\begin{cor}\label{cor1}
		There exists $k^*>1$ such that for $k \geq k^*$,
		\begin{align}
			\label{inv1}
			e^{\Lambda(t) (\P\la D\ran)^{1/\sigma}} e^{-\Lambda(t) (\P\la D\ran)^{1/\sigma}} &= I+R(t,x,D_x)\\
			\label{inv2}
			e^{-\Lambda(t) (\P\la D\ran)^{1/\sigma}} e^{\Lambda(t) (\P\la D\ran)^{1/\sigma}} &= I+\tilde R(t,x,D_x)
		\end{align}
		where $I+R$ and $I+\tilde R$ are invertible operators with $R, \tilde R \in C([0,T]; OPAG^{-\gamma e}_{\Phi,k,\sigma;\sigma,\sigma})$.	
	\end{cor}
	%\vspace{-0.5cm}
	\begin{proof}
		The equation (\ref{inv1}) can be derived by an application of Theorem \ref{conju} with $p(x,D) \equiv I \in OPAG^{0,0}_{\Phi,k;\sigma}$, where $I$ is an identity operator. This yields $R \in C([0,T]; OPAG^{-\gamma e}_{\Phi,k,\sigma;\sigma,\sigma})$. We can estimate the operator norm of $R(t,x,D_x)$ by $C_1k^{-\gamma}$.
		Choosing $k \geq k_1$, where $k_1$ is sufficiently large, ensures that the operator norm of $R(t,x,D_x)$ is strictly lesser than $1$. This guarantees the existence of 
		\[
		(I+R(t,x,D_x))^{-1} = \sum_{j=0}^{\infty} (-R(t,x,D_x))^{j}.
		\]
		As for the equation (\ref{inv2}), we follow the same procedure given in the proof of Theorem \ref{conju} with inequality (\ref{Pineq}) replaced with
		\begin{linenomath*}
			\begin{equation}
				\begin{aligned}
					-\P^{1/\sigma} + \Phi(x+\theta_1y)^{1/\sigma} &\leq -\P^{1/\sigma} + (\P+\Phi(\theta_1y))^{1/\sigma}\\
					&\leq -\P^{1/\sigma} + \P^{1/\sigma} + \Phi(\theta_1y))^{1/\sigma} \leq \Phi(y))^{1/\sigma},
				\end{aligned}
			\end{equation}
			for $x,y\in\R^n.$ This yields (\ref{inv2}) with $\tilde R \in C([0,T]; OPAG^{-\gamma e}_{\Phi,k,\sigma;\sigma,\sigma})$. Choosing $k \geq k_2$ for $k_2$ sufficiently large, guarantees that $I+\tilde R$ is invertible. Taking $k^* = \max\{k_1,k_2\}$ proves the corollary.
		\end{linenomath*}
		
	\end{proof}
	
	If we take $\Lambda(t) = \frac{\lambda}{\delta}(T^\delta - t^\delta)$ for $\lambda>0$ and $\delta \in (0,1)$, then we easily have  
	\begin{linenomath*}
		\begin{align*}
			e^{\Lambda(t) (\P\la D\ran)^{1/\sigma}} \: &\partial_t \: e^{-\Lambda(t) (\P\la D\ran)^{1/\sigma}}w(t,x) \\
			&= (I+R)\left(\partial_tw(t,x) -\Lambda'(t) (\P\la D\ran)^{1/\sigma}w(t,x)\right)\\
			& = (I+R)\Big(\partial_t + \frac{\lambda}{t^{1-\delta}}(\P\la D\ran)^{1/\sigma}\Big)w(t,x),
		\end{align*}
	\end{linenomath*}
	where the operator $R(t,x,D_x)$ is in $C([0,T]; OPAG^{-\gamma e}_{\Phi,k,\sigma;\sigma,\sigma})$. As in Corollary \ref{cor1},
	the operator $I+R(t,x,D_x)$ is invertible for sufficiently large $k$.
	In the proof of the main result, we choose $\lambda$ appropriately so that we can apply sharp G\r{a}rding inequality to prove the a priori estimate (\ref{est2}).
	
	\begin{cor}\label{cor2}
		Let $0 \leq \varepsilon \leq \varepsilon' < \Lambda_0$ where $\Lambda_0$ is as in Thereom \ref{conju}. Then
		\begin{linenomath*}
			\begin{equation}
				\label{inv3}
				e^{\varepsilon (\P\la D\ran)^{1/\sigma}} e^{-\varepsilon' (\P\la D\ran)^{1/\sigma}} = e^{(\varepsilon-\varepsilon') (\P\la D\ran)^{1/\sigma}}(I+\hat R(x,D_x)).
			\end{equation}
		\end{linenomath*}
		where $\hat R \in OPAG^{-\gamma e}_{\Phi,k,\sigma;\sigma,\sigma}$ and for sufficiently large $k$, $I+\hat R$ is invertible.
	\end{cor}
	\begin{proof}
		The equation (\ref{inv3}) can be derived by an easy extension of Theorem \ref{conju}. For this we replace the inequality (\ref{Pineq}) with
		\begin{linenomath*}
			\begin{equation}
				\label{Pineq2}
				\begin{aligned}
					\varepsilon\Phi(x)^{1/\sigma} - \varepsilon'\Phi(x+\theta_1y)^{1/\sigma} 
					&\leq \varepsilon\Phi(x)^{1/\sigma} - \varepsilon'(\P-\Phi(\theta_1y))^{1/\sigma} \\
					&\leq \varepsilon\Phi(x)^{1/\sigma} - \varepsilon'(\P^{1/\sigma}-\Phi(\theta_1 y)^{1/\sigma}) \\
					&\leq (\varepsilon-\varepsilon')\Phi(x)^{1/\sigma} +\varepsilon'\Phi(y)^{1/\sigma}
				\end{aligned}
			\end{equation}
		\end{linenomath*}
		when $\vert x \vert \geq \vert y \vert$ and $\varepsilon \P^{1/\sigma} - \varepsilon'\Phi(x+\theta_1y)^{1/\sigma} \leq \varepsilon \P^{1/\sigma} + \varepsilon' \Phi(y)^{1/\sigma} -\varepsilon' \P^{1/\sigma} $ when $\vert x \vert \leq \vert y \vert$
		and the inequality (\ref{xineq}) with
		\begin{linenomath*}
			\begin{equation}
				\label{xineq2}
				\begin{aligned}
					\varepsilon \la \xi+\eta+\zeta\ran^{1/\sigma} - \varepsilon'\japxin^{1/\sigma} 
					&\leq \varepsilon \japxin^{1/\sigma} + \varepsilon \la \eta+\zeta\ran^{1/\sigma} - \varepsilon'\japxin^{1/\sigma}\\
					& \leq (\varepsilon-\varepsilon') \japxin^{1/\sigma} + \varepsilon  \la \eta+\zeta\ran^{1/\sigma},
				\end{aligned}
			\end{equation}
			for $\xi,\eta,\zeta \in \R^n.$ 
		\end{linenomath*}
	\end{proof}
	We use the above corollaries to prove the continuity of an infinite order operator $e^{\varepsilon (\P\la D\ran)^{1/\sigma}}$ on the spaces $H^{s,\varepsilon',\sigma}_{\Phi,k}$.
	\begin{prop}
		The operator $e^{\varepsilon (\P\la D\ran)^{1/\sigma}} : H^{s,\varepsilon',\sigma}_{\Phi,k} \to H^{s,\varepsilon'-\varepsilon,\sigma}_{\Phi,k}$ is continuous for $k\geq k_0$ and $0 \leq \varepsilon \leq \varepsilon' < \Lambda_0$ where $k_0$ sufficiently large and $\Lambda_0$ is as in Thereom \ref{conju}.
	\end{prop}
	\begin{proof}
		Consider $w $ in $H^{s,\varepsilon',\sigma}_{\Phi,k} $. From Corollaries \ref{cor1} and \ref{cor2}, we have
		$$
		\begin{aligned}
			e^{-\varepsilon' (\P\la D\ran)^{1/\sigma}} e^{\varepsilon' (\P\la D\ran)^{1/\sigma}} &= I+R_1(x,D_x), \\
			e^{\varepsilon (\P\la D\ran)^{1/\sigma}} e^{-\varepsilon' (\P\la D\ran)^{1/\sigma}} &= e^{(\varepsilon-\varepsilon') (\P\la D\ran)^{1/\sigma}}(I+R_2(x,D_x)),\\
			e^{(\varepsilon' -\varepsilon) (\P\la D\ran)^{1/\sigma}} e^{-(\varepsilon' -\varepsilon) (\P\la D\ran)^{1/\sigma}} &= I+R_3(x,D_x).
		\end{aligned}
		$$	
		where $R_1, R_2 , R_3\in OPAG^{-\gamma e}_{\Phi,k,\sigma;\sigma,\sigma}$.	For $k\geq k_0$, $k_0$ sufficiently large, the operators $I+R_j(x,D_x),j=1,2,3$ are invertible.
		Then, one can write 
		\begin{equation*}
			%\label{c1}
			e^{\varepsilon (\P\la D\ran)^{1/\sigma}} w = e^{\varepsilon (\P\la D\ran)^{1/\sigma}}  \left(e^{-\varepsilon' (\P\la D\ran)^{1/\sigma}} e^{\varepsilon' (\P\la D\ran)^{1/\sigma}} -R_1 \right) w.
		\end{equation*}
		This implies that 
		\begin{equation}
			\label{c2}
			e^{\varepsilon (\P\la D\ran)^{1/\sigma}} (I+R_1)w = e^{(\varepsilon-\varepsilon') (\P\la D\ran)^{1/\sigma}} (I+R_2) e^{\varepsilon' (\P\la D\ran)^{1/\sigma}} w.
		\end{equation}
		From (\ref{c2}), we have
		\[
		e^{(\varepsilon' -\varepsilon) (\P\la D\ran)^{1/\sigma}} e^{\varepsilon (\P\la D\ran)^{1/\sigma}} (I+R_1)w = (I+R_3)(I+R_2) e^{\varepsilon' (\P\la D\ran)^{1/\sigma}} w.
		\]
		Note that $(I+R_j),j=1,2,3,$ are bounded and invertible operators. Substituting  $w = (I+R_1)^{-1}v$ and taking $L^2$ norm on both sides of the above equation yields 
		\[
		\Vert e^{\varepsilon (\P\la D\ran)^{1/\sigma}} v \Vert_{\Phi,k;s,\varepsilon'-\varepsilon,\sigma} \leq C_1 \Vert (I+R_1)^{-1}v  \Vert_{\Phi,k;s,\varepsilon',\sigma} \leq C_2 \Vert v  \Vert_{\Phi,k;s,\varepsilon',\sigma},
		\]
		for all $v \in H^{s,\varepsilon',\sigma}_{\Phi,k}$ and for some $C_1,C_2>0$. This proves the proposition. 
	\end{proof}
	
	Note that in the proof of Theorem \ref{result2} we perform a change of variable (\ref{cv}) where we require that $\Lambda(t)< \Lambda^* = \min\{\Lambda_0,\Lambda_1,\Lambda_2\}$. Here $\Lambda_0$ is as dictated by Theorem \ref{conju} while $\Lambda_1$ and $\Lambda_2$ are Sobolev indices of the Cauchy data and the right hand side of the Cauchy problem (\ref{eq1}), respectively.

		\section{Subdivision of the Phase Space}\label{zones}
	One of the main tools in our analysis is the division of the extended phase space into two regions using the Planck function, $h(x,\xi)=(\P \japxin)^{-1}$. We use these regions in the proof of Theorem \ref{result2} (see Section \ref{factr}) to handle the low regularity in $t$. To this end we define $t_{x,\xi}$, for a fixed $(x,\xi)$, as the solution to the equation
	\begin{linenomath*}
		\[
		t^{q}=N\:h(x,\xi),
		\]
	\end{linenomath*}
	where $N$ is the positive constant and $q$ is the given order of singularity. 
	Since $3 \leq  \sigma < q/(q-1)$, we consider $\delta \in (0,1)$ such that
	\begin{linenomath*}
		\begin{equation} \label{delta}
			\frac{1}{\sigma}=\frac{q-1+\delta}{q}=1-\frac{1-\delta}{q}.
		\end{equation}
	\end{linenomath*}
	Denote $\gamma=1 -\frac{1}{\sigma}$. Using $t_{x,\xi}$ and the notation $J=\J$ we define the interior region
	\begin{linenomath*}
		\begin{equation} \label{zone1}
			\begin{aligned}
				\pd &=\{(t,x,\xi)\in J : 0 \leq t \leq t_{x,\xi}, \: |x|+|\xi|> N\} \\
				&=\{(t,x,\xi) \in J :t^{1-\delta}\leq N^\gamma\:\ h(x,\xi)^{\gamma},\: |x|+|\xi|>N\},
			\end{aligned}
		\end{equation}
	\end{linenomath*}
	and the exterior region
	\begin{linenomath*}
		\begin{equation} \label{zone2}
			\begin{aligned}
				\hyp &=\{(t,x,\xi)\in J : t_{x,\xi} \leq t \leq T, \: |x|+|\xi|> N\} \\
				&=\{(t,x,\xi)\in J : t^{1-\delta}\geq N^\gamma\:\ h(x,\xi)^{\gamma},\: |x|+|\xi|> N\}.
			\end{aligned}
		\end{equation}
	\end{linenomath*}	
	The utility of these regions lies in decomposing our operator into, mainly, two operators. The first operator has a high-order in $(x,\xi)$ but excludes the singularity at $t=0$ and the second operator has a singularity at $t=0$ but is of lower-order in $(x,\xi)$.

			\section{Proof of the Main Result} \label{Proofs}
	In this section, we give a proof of the main result. %The key steps in our proof are similar to some of the techniques used in \cite{Cico1,Asc,CicoLor,AscaCappi2}.
	There are three key steps in the proof of Theorem \ref{result2}. First, we factorize the operator $P(t,x,\partial_t,D_x)$. To this end, we begin with regularizing the characteristic roots of the principal symbol of the operator. Second, we reduce the operator $P$ to a pseudodifferential system of first order.  Lastly, we perform a conjugation to deal with the low-regularity in $t$. Using sharp G\r{a}rding’s inequality we arrive at the $L^2$-well-posedness of a related auxiliary Cauchy problem, which gives well-posedness of the original problem in the Sobolev spaces $H^{s,\varepsilon,\sigma}_{\Phi,k}$, $3 \leq \sigma<q/(q-1)$.

	\subsection{Factorization}\label{factr}
	We are interested in a factorization of the operator $P(t,x,\partial_t,D_x)$. Formally, this leads to
	\begin{linenomath*}
		\begin{equation}\label{fact}	
			\begin{aligned}		
				P(t,x,\partial_t,D_x)=(&\partial_t-i\tau_m(t,x,D_x))\cdots(\partial_t-i\tau_1(t,x,D_x))\\
				&+\sum_{k=0}^{m-1}R_k(t,x,D_x)\partial_t^k
			\end{aligned}
		\end{equation}
	\end{linenomath*}
	where $\tau_j \in C([0,T];AG^e_{\Phi,k;\sigma})\cap C^1((0,T];AG^e_{\Phi,k;\sigma})$ such that $t^q\partial_t\tau_j \in C([0,T];AG^e_{\Phi,k;\sigma})$. Since the operators $\tau_j(t,x,D_x)$ are not differentiable with respect to $t$ at $t=0$, we use regularized roots $\lambda_j(t,x,D_x)$ in (\ref{fact}) instead of $\tau_j(t,x,D_x)$ for $j=1,\cdots, m$. For this purpose we extend the roots on $(T,\infty]$ by setting
	\begin{linenomath*}
		\begin{align*}
			%\tau_j \var &= \tau_j(0,x,\xi) \; \text{ when } t<0, \\
			\tau_j \var &= \tau_j(T,x,\xi) \; \text{ when } t>T.
		\end{align*}
	\end{linenomath*}
	Then we define the regularized root $\lambda_j(t,x,\xi)$ as 
	\begin{linenomath*}
		\begin{equation}
			\label{molli}
			\lambda_j(t,x,\xi)=\int_\R \tau_j\left(t- h(x,\xi)s,x,\xi\right)\rho(s)ds
		\end{equation}
	\end{linenomath*}
	where $\rho$ is compactly supported smooth function in $ \mathcal{S}_1^1(\R)$ satisfying $\int\limits_{\R}\rho(s)ds=1$ and $0\leq\rho(s)\leq1$ with supp $\rho(s) \subset \R_{<0}$. Then
	\begin{linenomath*}
		\begin{align*}
			(\lambda_j-\tau_j)(t,x,\xi) &= \int (\tau_j(t- h(x,\xi)s,x,\xi)-\tau_j(t,x,\xi))\rho(s)ds\\
			&= \frac{1}{\h}\int (\tau_j(s,x,\xi)-\tau_j(t,x,\xi))\rho((t-s)h(x,\xi)^{-1})ds.
		\end{align*}
	\end{linenomath*}	
	It is easy to see that
	\begin{linenomath*}
		\begin{equation}
			\label{intg}
			\begin{cases}
				\lambda_j-\tau_j \in L^1([0,T];AG^e_{\Phi,k;\sigma})\cap C([0,T];AG^e_{\Phi,k;\sigma})\\
				\partial_t^k\lambda_j \in L^1([0,T];AG^{(k+1)e}_{\Phi,k;\sigma})\cap C([0,T];AG^{(k+1)e}_{\Phi,k;\sigma}),
			\end{cases} 
		\end{equation}
	\end{linenomath*}
	and note that in $\hyp$ we have
	\begin{linenomath*}
		\begin{equation}
			\label{notintg}
			\begin{cases}
				t^q(\lambda_j-\tau_j) \in C([0,T];AG^{0,0}_{\Phi,k;\sigma}) \\
				t^q\partial_t^k\lambda_j \in C([0,T];AG^{ke}_{\Phi,k;\sigma}).
			\end{cases}
		\end{equation}
	\end{linenomath*}	
	We define the operator 
	\begin{linenomath*}
		\[
		\tilde{P}(t,x,\partial_t,D_x)=(\partial_t-i\lambda_m(t,x,D_x)) \cdots (\partial_t-i\lambda_1(t,x,D_x)).
		\]
	\end{linenomath*}
	By (\ref{intg}) and (\ref{notintg}) one has the following factorization of the operator $P$
	\begin{linenomath*}
		\[
		P(t,x,\partial_t,D_x)=\tilde{P}(t,x,\partial_t,D_x)+R(t,x,\partial_t,D_x),
		\]
	\end{linenomath*}
	where
	\begin{linenomath*}
		\[
		R(t,x,\partial_t,D_x)=\sum_{k=0}^{m-1}R_j(t,x,D_x)\partial_t^j
		\]
	\end{linenomath*}
	such that for $j=0,\dots,m-1$,
	\begin{linenomath*}
		\begin{equation}
			\label{pd}
			R_j \in L^1([0,T];AG^{(m-j)e}_{\Phi,k;\sigma})\cap C([0,T];AG^{(m-j)e}_{\Phi,k;\sigma}), \:\; \text{ and }
		\end{equation}
	\end{linenomath*}
	\begin{linenomath*}
		\begin{equation}
			\label{hyp}
			t^q R_j \in C([0,T];AG^{(m-1-j)e}_{\Phi,k;\sigma}) \text{ in } \hyp.
		\end{equation}
	\end{linenomath*}	
	To determine the precise Gevrey regularity for $R_{j}(t,x,\xi)$, we consider the regions, $\pd$ and $\hyp$, separately.	
	In $\pd$, we have $(\P\japxin)^\gamma \leq \frac{N^\gamma}{t^{1-\delta}}
	$. Using this and (\ref{pd}), we can write 
	\begin{linenomath*}
		\begin{eqnarray}
			\nonumber \vert \partial_{\xi}^{\alpha} D_x^\beta R_j(t,x,\xi) \vert &\leq& C^{|\alpha|+|\beta|} \beta!^\sigma \alpha! \P^{m-\gamma-j-|\beta|}\japxin^{m-\gamma-j-|\alpha|} (\P\japxin)^\gamma \\
			\nonumber &\leq & C_1^{|\alpha|+|\beta|} \beta!^\sigma \alpha!\frac{N^\gamma}{t^{1-\delta}}\P^{m-\gamma-j-|\beta|} \japxin^{m-\gamma-j-|\alpha|}
		\end{eqnarray}
	\end{linenomath*}
	Similarly, in $\hyp$, we have $t^{q/\sigma} \geq (N\;h(x,\xi))^{\frac{1}{\sigma}} $ and
	\[
	\frac{1}{t^q}=\frac{1}{t^{1-\delta}}\frac{1}{t^{q/\sigma}} \leq \frac{1}{t^{1-\delta}}\bigg(\frac{\P\japxin}{N}\bigg)^{1/\sigma}
	\]
	Using this and (\ref{hyp}), we have 
	\begin{linenomath*}
		\begin{align*}
			\nonumber \vert \partial_{\xi}^{\alpha}D_x^\beta R_j(t,x,\xi) \vert &\leq C^{|\alpha|+|\beta|} \beta!^\sigma \alpha!\frac{1}{t^q} \P^{m-1-j-|\beta|} \japxin^{m-1-j-|\alpha|}\\
			\nonumber &\leq C_1^{|\alpha|+|\beta|} \beta!^\sigma \alpha!\frac{N^{-1/\sigma}}{t^{1-\delta}}\P^{m-\gamma-j-|\beta|} \japxin^{m-\gamma-j-|\alpha|} 
		\end{align*}
	\end{linenomath*}
	Hence, we have 
	$
	t^{1-\delta}R_j \in C([0,T];AG_{\Phi,k;\sigma}^{(m-\gamma-j)e})
	$ for $j=0,\cdots,(m-1).$
	
	\subsection{Reduction to First Order System}
	We will now reduce the operator $P$ to an equivalent first order pseudodifferential system. The procedure is similar to the one used in \cite[Section 4.2 \& 4.3]{CicoLor}. To achieve this, we introduce the change of variables $U=U(t,x)=(u_0(t,x),\cdots,u_{m-1}(t,x))^T$, where
	\begin{linenomath*}
		\begin{equation*}
			%\label{COV}
			\begin{cases}
				u_0(t,x)=\P^{m-1} \langle D_x \ran^{m-1}u(t,x), \\ 
				u_j(t,x)=\P^{m-1-j} \langle D_x \ran^{m-1-j}(\partial_t-i\lambda_j(t,x,D_x)) \cdots (\partial_t-i\lambda_1(t,x,D_x))u(t,x), \\  
			\end{cases}
		\end{equation*}
	\end{linenomath*}
	for $j=1,\cdots m-1$. Then, $Pu=f$ is equivalent to	
	\begin{linenomath*}
		\[
		(\partial_t-A_1(t,x,D_x)+A_2(t,x,D_x))U(t,x) = F(t,x),
		\]
	\end{linenomath*}
	where $F(t,x)=(0,\dots,0,f(t,x))^T$,
	\begin{linenomath*}
		\[
		A_1\left(t, x, D_{x}\right)=\left(\begin{array}{ccccc} i \lambda_{1}\left(t, x, D_{x}\right) & \P \la D_{x}\ran & 0 & \ldots & 0 \\ 0 & \ddots & \ddots & & \vdots \\ \vdots & & \ddots & \ddots & 0 \\ \vdots & & & \ddots & \P\la D_{x}\ran \\ 0 & \ldots & \ldots & 0 & i\lambda_{m}\left(t, x, D_{x}\right)\end{array}\right),
		\]
	\end{linenomath*}
	and $A_2(t,x,D_x) = \{a^{(2)}_{i,j}(t,x,D_x)\}_{1 \leq i,j \leq m}$ is a matrix of lower order terms with $a^{(2)}_{i,j} \in AG^{0,0}_{\Phi,k;\sigma}$ for $i=1, \dots,m-1$ and $j=1,\dots,m$, and $t^{1-\delta}a^{(2)}_{m,j} \in C([0,T];AG_{\Phi,k;\sigma}^{\frac{1}{\sigma}e})$ for $j=1,\dots,m$.
	
	Consider the $T(t,x,\xi)=\{\beta_{p,q}(t,x,\xi)\}_{0 \leq p,q \leq m-1}$, where
	\begin{linenomath*}
		\begin{align*}
			\beta_{p,q}(t,x,\xi) &= 0, \qquad p \geq q; \\
			\beta_{p,q}(t,x,\xi) &= \frac{(1-\varphi_1(\P\japxi))(\P \japxin)^{q-p}}{d_{p,q}(t,x,\xi)}, \qquad p < q;\\
			d_{p,q}(t,x,\xi) &= \prod_{r=p+1}^{q}i\Big(\lambda_{q+1}(t,x,\xi)-\lambda_{r}(t,x,\xi)\Big),
		\end{align*}
	\end{linenomath*}
	where $\varphi_1 \in C_0^\infty(\R)$, $\varphi_1(r)=1$ for $|r| \leq M$, for a large parameter $M$. Note that the matrix $T(t,x,\xi)$ is nilpotent. We define $H(t,x,D_x)$ and $\tilde H(t,x,D_x)$ to be pseudodifferential operators with symbols
	\begin{linenomath*}
		\begin{align}
			H(t,x,\xi) &=I+T(t,x,\xi),  \; \text{and} \nonumber \\
			\label{H}
			\tilde H(t,x,\xi) &=I+\sum_{j=1}^{m-1}(-1)^jT^j(t,x,\xi),
		\end{align}
	\end{linenomath*}
	respectively.
	\begin{prop}
		For the operators $H(t,x,D_x)$ and $\tilde H(t,x,D_x)$, the following assertions hold true
		\begin{enumerate}[label=\roman*)]
			\item $H(t,x,D_x)$ and $\tilde H(t,x,D_x)$ are in $C\Big([0,T];OPAG_{\Phi,k;\sigma}^{0,0}\Big)$.
			\item The composition $(H \circ \tilde H)(t,x,D_x)$ satisfies
			\begin{linenomath*}
				\begin{equation}\label{nilop}
					H(t,x,D_x) \circ \tilde H(t,x,D_x)=I+K(t,x,D_x)
				\end{equation}
			\end{linenomath*}
			where $K(t,x,D_x) \in C\Big([0,T];OPAG_{\Phi,k;\sigma}^{-e}\Big)$.
			\item The operator $t^{q}(\partial_tH)(t,x,D_x)$ belongs to $C\Big([0,T];OPAG_{\Phi,k;\sigma}^{0,0})\Big)$.
			\item The operator $t^{1-\delta}(\partial_tH)(t,x,D_x)$ belongs to $C\Big([0,T];OPAG_{\Phi,k;\sigma}^{\frac{1}{\sigma}e})\Big)$.
		\end{enumerate}
	\end{prop}
	Note that $I + K(t,x,D_x)$ is invertible for sufficiently large $M$. Since $K(t,x,D_x) \in C\Big([0,T];OPAG_{\Phi,k;\sigma}^{-e}\Big)$, we can estimate the operator norm of $K$ by $CM^{-1}$ where $M$ is as in the definition of $H(t,x,D_x)$. We choose $M$ sufficiently large so that the operator norm of $K(t,x,D_x)$ is strictly lesser than $1$. This implies that $I+K(t,x,D_x)$ is invertible and 
	\begin{linenomath*}
		\[
		(I+K(t,x,D_x))^{-1} = \sum_{j=0}^{\infty} (-K(t,x,D_x))^j \in C\Big([0,T];OPAG_{\Phi;\sigma}^{0,0}\Big).
		\]
	\end{linenomath*}
	We perform the described change of variable by setting
	\begin{linenomath*}
		\[\hat{U}=\hat{U}(t,x) = \tilde H (t,x,D_x)U(t,x).\]
	\end{linenomath*} 
	The above equation and (\ref{nilop}) imply that 
	$$
	U(t,x) = (I+K(t,x,D_x))^{-1}H(t,x,D_x)\hat{U}(t,x).
	$$
	Noting this fact, we obtain (similar to \cite[Section 4.3]{CicoLor}) the first order system equivalent to (\ref{eq1}) :
	\begin{linenomath*}
		\[
		(\partial_t-A_3(t,x,D_x)+A_4(t,x,D_x)) \hat{U}(t,x) = F_1(t,x),
		\]
	\end{linenomath*}
	where $F_1(t,x)=\mathcal{K}(t,x,D_x)^{-1}\tilde H(t,x,D_x)\mathcal{K}(t,x,D_x)F(t,x)$ for $\mathcal{K}(t,x,D_x) = (I+K(t,x,D_x))$ and the operators $A_3$ and $A_4$ are as follows
	\begin{linenomath*}
		\begin{align*}
			A_3 &= \mathcal{K}^{-1} \tilde H \mathcal{K} A_1\mathcal{K}^{-1} H,\\
			A_4 &=  \mathcal{K}^{-1} \tilde H \mathcal{K} A_2 \mathcal{K}^{-1} H + \mathcal{K}^{-1} \tilde H \mathcal{K} (\partial_t\mathcal{K}^{-1}) H+ \mathcal{K}^{-1} \tilde H \partial_t H.
		\end{align*}
	\end{linenomath*}
	We can write 
	\begin{linenomath*}
		\begin{align*}
			A_3(t,x,D_x) &= \mathcal{D}(t,x,D_x)  + \tilde A_3(t,x,D_x) \\
			A(t,x,D_x)  & = A_4(t,x,D_x) - \tilde A_3(t,x,D_x)
		\end{align*}
	\end{linenomath*}
	where $\mathcal{D}=\text{diag}(i\lambda_1(t,x,D_x),\ldots,i\lambda_m(t,x,D_x))$, and  $A(t,x,D_x)$ contains the lower order terms whose symbol is such that 
	\begin{linenomath*} 
		\[
		t^{1-\delta}A \in C([0,T];{AG}_{\Phi,k;\sigma}^{\frac{1}{\sigma}e}).
		\]
	\end{linenomath*}
	Then, $Pu = f$ is equivalent to
	\begin{linenomath*}
		\begin{equation}
			\label{FOS1}
			L_1\hat{U}=(\partial_t - \mathcal{D}+A)\hat{U}= F_1(t,x).
		\end{equation}
	\end{linenomath*}
	
	We prove the a priori estimate (\ref{est2}) by proving that
	\begin{linenomath*}
		\begin{equation*}        
			\Vert \hat{U}(t)\Vert_{\Phi,k;s,\Lambda(t),\sigma} \leq C\Big(\Vert \hat{U}(0)\Vert_{\Phi,k;s,\Lambda(0),\sigma} + \int_{0}^{t} \Vert F_1(\tau,\cdot)\Vert_{\Phi,k;s,\Lambda(\tau),\sigma} d\tau\Big),
		\end{equation*}
	\end{linenomath*}
	where $\Lambda(t)=\frac{\lambda}{\delta}(T^\delta-t^\delta)$ for sufficiently large $\lambda$. 
	
	It is sufficient to prove the above estimate for $s=(0,0)$ since the operator $L_2=\P^{s_2}\la D_x \ran^{s_1} L_1\la D_x \ran^{-s_1}\P^{-s_2}$ satisfies the same hypotheses as $L_1$. That is, for $V(t,x)=\P^{s_2}\la D_x \ran^{s_1}\hat{U}(t,x)$, $L_1\hat{U}=F_1$ implies $L_2V=F_2$ where $F_2(t,x)=\P^{s_2}\la D_x \ran^{s_1}F_1(t,x)$. So, assuming $s=(0,0)$ we let $L_2=L_1= \partial_t-\mathcal{D}+A$.
	
	To deal with the low-regularity in $t$, we introduce the following change of variable
	\begin{linenomath*}
		\begin{equation}
			\label{cv}
			W(t,x) = e^{\Lambda(t)(\P\la D_x\ran)^{1/\sigma}} V(t,x).
		\end{equation}
	\end{linenomath*}
	This implies that $V(t,x)= (I+\tilde R(t,x,D_x))^{-1}e^{-\Lambda(t)(\P\la D_x\ran)^{1/\sigma}}W(t,x)$
	where $I+\tilde R(t,x,D_x)$ is an invertible operator as in Corollary \ref{cor1}. Let us denote $I+\tilde R(t,x,D_x)$ and $e^{\pm\Lambda(t)(\P\la D_x\ran)^{1/\sigma}}$ by $\mathcal{R}(t,x,D_x)$ and $E^{(\pm)}(t,x,D_x)$, respectively. Then $Pu = f $ is equivalent to $L_3W=F_3$ where
	\[
	L_3 = \partial_t-\mathcal{D} + \Big(B+\frac{\lambda}{t^{1-\delta}}(\P\la D_x\ran)^{1/\sigma}\Big),
	\]
	$F_3(t,x) = \mathcal{R}^{-1}E^{(+)}\mathcal{R} F_2(t,x)$ and the operator $B(t,x,D_x)$ is given by
	\[
	B = \mathcal{R}^{-1}E^{(+)} \Big( \mathcal{R}(\partial_t \mathcal{R}^{-1}) + \mathcal{R} A\mathcal{R}^{-1} \Big) E^{(-)} - \big( \mathcal{R}^{-1}E^{(+)}\mathcal{R} \mathcal{D}\mathcal{R}^{-1}E^{(-)} - \mathcal{D}  \big).
	\]
	Observe that from Theorem \ref{conju} and from the Cauchy data given in conditions (ii) and (iii) of Theorem \ref{result2}, we need 
	$\Lambda(t) < \Lambda^*=\min\{\Lambda_0,\Lambda_1,\Lambda_2\}.$
	This implies $T < \big(\frac{\delta}{\lambda} \Lambda^*\big)^{1/\delta}.$ Then, we have $t^{1-\delta}B \in C([0,T];{OPAG}_{\Phi,k,\sigma;\sigma,\sigma}^{\frac{1}{\sigma}e})$.
	The estimate (\ref{est2}) on the solution $u$ can be established by proving that the function $W(t,x)$
	satisfies the a priori estimate
	\begin{linenomath*}
		\begin{equation}
			\label{est3}
			\Vert W(t) \Vert^2_{L^2} \leq C \Big(\Vert W(0) \Vert^2_{L^2} + \int_{0}^{t}\Vert F_3(\tau,\cdot)\Vert_{L^2}d\tau \Big) , \quad t\in[0,T], \; C>0.
		\end{equation}   
	\end{linenomath*}
	
	\subsection{Well-posedness of the First Order System}\label{sgi}	
	Observe that we have $ L_3W=F_3$ whenever $L_1\hat{U}=F_1$ and
	$
	\Vert W(t)\Vert_{L^2}=\Vert \hat{U}(t) \Vert_{\Phi,k;s,\Lambda(t),\sigma}.
	$
	Moreover, the problem $L_3W=F_3$ is equivalent to an auxiliary problem
	\begin{linenomath*}
		\begin{equation*}
			\partial_tW=\mathcal{D}W-\Big(\frac{\lambda}{t^{1-\delta}}(\P\la D_x\ran)^{1/\sigma}W+BW\Big) + F_3(t,x),
		\end{equation*}
	\end{linenomath*}
	for $(t,x) \in (0,T] \times \R^n$, with initial conditions
	\begin{linenomath*}
		\[
		W(0,x)=(w_0(x),\dots,w_{m-1}(x))^T, \quad \text{where}
		\]
	\end{linenomath*}    
	\begin{linenomath*}
		\begin{align*}
			w_j(x) = &e^{\Lambda(0)(\P\la D_x\ran)^{1/\sigma}}\P^{s_2} \langle D_x \rangle^{s_1} \tilde H(0,x,D_x)
			\P^{m-1-j} \langle D_x \rangle^{m-1-j}\\
			&\times (\partial_t-i\lambda_j(0,x,D_x)) \cdots (\partial_t-i\lambda_1(0,x,D_x))u(0,x)   	
		\end{align*}
	\end{linenomath*}
	for $j=0,\dots,m-1$.	
	To prove (\ref{est3}), let us consider 
	\begin{linenomath*}
		\begin{align}
			\partial_t\Vert W(t) \Vert^2_{L^2} & = 2\Re\langle \partial_tW,W \rangle_{L^2} \nonumber\\	
			&= 2\Re\langle\mathcal{D}W,W \rangle_{L^2}-2\Re\Big\langle\Big(\frac{\lambda}{t^{1-\delta}}(\P\la D_x\ran)^{1/\sigma}+B\Big)W,W \Big\rangle_{L^2} \nonumber \\
			\label{fin1}
			& \quad + 2\Re\la F_3,W\ra.
		\end{align}	
	\end{linenomath*}
	Since $\mathcal{D}$ is diagonal with purely imaginary entries, we have
	\begin{linenomath*}
		\begin{equation}
			\label{neg1}
			\Re\langle\mathcal{D}W,W \rangle_{L^2}  \leq C_1 \Vert W(t)\Vert_{L^2} .
		\end{equation}
	\end{linenomath*}	
	Also, note that $t^{1-\delta}h(x,D_x)^{1/\sigma} B(t,x,D_x) \in C([0,T];{OPAG}_{\Phi,\sigma;\sigma,\sigma}^{0,0})$. We choose $\lambda$ sufficiently large so that we can apply sharp G\r{a}rding inequality, see \cite[Theorem 18.6.14]{Horm}, for the metric $\tilde{g}_{\Phi,k}$ given in (\ref{om2}) with the Planck function $\tilde{h}(x,\xi)=(\P\japxin)^{\frac{1}{\sigma}-\gamma}$. It is important to note that the application of sharp G\r{a}rding inequality requires $ \sigma \geq 3.$ This yields
	\begin{linenomath*}
		\begin{equation}
			\label{neg2}
			\Re\Big\langle\Big(\frac{\lambda}{t^{1-\delta}}(\P\la D_x\ran)^{1/\sigma}+B\Big)W,W \Big\rangle_{L^2} \geq -C_2 \Vert W \Vert_{L^2}, \quad C_2>0.
		\end{equation}
	\end{linenomath*}
	From (\ref{fin1}), (\ref{neg1}) and (\ref{neg2}) we have 
	\begin{linenomath*}
		\[
		\frac{d}{dt}\Vert W(t) \Vert^2_{L^2} \leq C \Vert W(t)\Vert_{L^2} + C \Vert F_3(t,\cdot)\Vert_{L^2}. 
		\]
	\end{linenomath*}
	
	Considering the above inequality as a differential inequality, we apply Gronwall's lemma and obtain that
	\begin{linenomath*}
		\[
		\Vert W(t)\Vert_{L^2}^2\leq C \Vert W(0)\Vert^2_{L^2} + C \int_{0}^{t} \Vert F_3(\tau,\cdot)\Vert^2_{L^2} d\tau.
		\]	
	\end{linenomath*}
	This proves the well-posedness of the auxiliary Cauchy problem. Note that the solution $\hat{U}$ to (\ref{FOS1}) belongs to $C([0,T];H^{s,\Lambda(t),\sigma}_{\Phi,k})$. Returning to our original solution $u=u(t,x)$ we obtain the estimate (\ref{est2}) with 
	\begin{linenomath*}
		\[
		u \in \bigcap\limits_{j=0}^{m-1}C^{m-1-j}([0,T];H^{s+ej,\Lambda(t),\sigma}_{\Phi,k}).
		\]	
	\end{linenomath*}	
	This concludes the proof.	
	
	Next, we use the a priori estimate derived in this section to obtain a cone condition for the Cauchy problem.

		\section{Cone Condition}\label{cone}
	Let $K(x^0,t^0)$ denote the cone with the vertex $(x^0,t^0)$:
	\begin{linenomath*}
		\[
		K(x^0,t^0)= \{(t,x) \in [0,T] \times \R^n : |x-x^0| \leq c\P(t-t^0)\},
		\]
	\end{linenomath*}
	where $c>0$. The cone $K(x^{0}, t^{0})$ has a slope $c\P$ which governs the speed of the growth of the cone. Note that the speed is anisotropic, that is, it varies with $x$. For a given $(t,x)$, it is well known in the literature (see \cite[Section 3.11]{yag}) that the influence of the vertex of cone is carried farther by the dominating characteristic root of the principal symbol of the operator $P$ in (\ref{eq1}). So, the constant $c$ is determined by the characteristic roots as 
	\begin{linenomath*}
		\begin{equation}
			\label{speed}
			c= \sup\Big\{|\tau_k(t,x,\xi)|\P^{-1} : k=1,\dots,m\Big\},
		\end{equation}
	\end{linenomath*}
	where supremum is taken over the set $\{(t,x,\xi) \in[0,T] \times \R^n_x \times \R^n_\xi:|\xi|=1\}$.
	For example, in case of a second order strictly hyperbolic operator given in (\ref{model}), we have 
	\begin{linenomath*}
		\[
		c= \sup_{t\in[0,T],x \in \R^n}\sqrt{a(t,x)}\P^{-1}.
		\]
	\end{linenomath*}
	Note that the speed of growth of the cone increases as $|x|$ increases since $\Phi$ is monotone increasing function of $|x|$. In the following we prove the cone condition for the Cauchy problem $(\ref{eq1})$.
	
	\begin{prop}
		The Cauchy problem (\ref{eq1}) has a cone dependence, that is, if 
		\begin{linenomath*}
			\begin{equation}\label{cone1}
				f\big|_{K(x^0,t^0)}=0, \quad f_i\big|_{K(x^0,t^0) \cap \{t=0\}}=0, \; i=1, \dots, m
			\end{equation}
		\end{linenomath*}
		then
		\begin{linenomath*}
			\begin{equation}\label{cone2}
				u\big|_{K(x^0,t^0)}=0,
			\end{equation}
		\end{linenomath*}
		provided that $c$ is as in (\ref{speed}).
	\end{prop}
	\begin{proof}
		Consider $t^0>0$, $c>0$ and assume that  (\ref{cone1}) holds. We define a set of operators $P_\varepsilon(t,x,\partial_t,D_x), 0 \leq \varepsilon \leq \varepsilon_0$ by means of the operator $P(t,x,\partial_t,D_x)$ in (\ref{eq1}) as follows
		\begin{linenomath*}
			\[
			P_\varepsilon(t,x,\partial_t,D_x) = P(t+\varepsilon,x,\partial_t,D_x), \: t \in [0,T-\varepsilon_0], x \in \R^n,
			\]
		\end{linenomath*}
		and $\varepsilon_0 < T-t^0$, for a fixed and sufficiently small $\varepsilon_0$. For these operators we consider Cauchy problems
		\begin{linenomath*}
			\begin{alignat}{2}
				P_\varepsilon v_\varepsilon & =f,  &&  t \in [0,T-\varepsilon_0], \; x \in \R^n\\
				\partial_t^{k-1}v_\varepsilon(0,x)& =f_k(x),\qquad && k=1,\dots,m.
			\end{alignat}
		\end{linenomath*}
		Note that $v_\varepsilon(t,x)=0$ in $K(x^0,t^0)$ and $v_\varepsilon$ satisfies the a priori estimate (\ref{est2}) for all $t \in[0,T-\varepsilon_0]$. Further, we have 
		\begin{linenomath*}
			\begin{alignat}{2}
				P_{\varepsilon_1} (v_{\varepsilon_1}-v_{\varepsilon_2}) & = (P_{\varepsilon_2}-P_{\varepsilon_1})v_{\varepsilon_2},\qquad  &&  t \in [0,T-\varepsilon_0], \; x \in \R^n\\
				\partial_t^{k-1}(v_{\varepsilon_1}-v_{\varepsilon_2})(0,x)& = 0,\qquad && k=1,\dots,m.
			\end{alignat}
		\end{linenomath*}
		For the sake of simplicity we denote $b_{j,\alpha}(t,x)$, the coefficients of lower order terms, as $a_{j,\alpha}$ for $j+|\alpha|<m$. Substituting $s-e$ for $s$ in the a priori estimate, we obtain
		\begin{linenomath*}
			\begin{equation}\label{cone3}
				\begin{aligned}
					&\sum_{j=0}^{m-1} \Vert  \partial_t^j(v_{\varepsilon_1}-v_{\varepsilon_2})(t,\cdot) \Vert_{\Phi,k;s+(m-2-j)e,\Lambda(t),\sigma} \\
					&\leq C \int_{0}^{t}\Vert (P_{\varepsilon_2}-P_{\varepsilon_1})v_{\varepsilon_2}(\tau,\cdot)\Vert_{\Phi,k;s-e,\Lambda(\tau),\sigma}\;d\tau\\
					&\leq C \int_{0}^{t} \sum\limits_{\substack{j+|\alpha|=m \\ j<m}}\Vert (a_{j,\alpha}(\tau+\varepsilon_1,x) - a_{j,\alpha}(\tau+\varepsilon_2,x))\partial_t^jD_x^\alpha v_{\varepsilon_2}(\tau,\cdot)\Vert_{\Phi,k;s-e,\Lambda(\tau),\sigma}\;d\tau	.
				\end{aligned}
			\end{equation}
		\end{linenomath*}
		Using the Taylor series approximation, we have
		\begin{linenomath*}
			\begin{align*}
				|a_{j,\alpha}(\tau+\varepsilon_1,x) - a_{j,\alpha}(\tau+\varepsilon_2,x)| &= \Big|\int_{\tau+\varepsilon_2}^{\tau+\varepsilon_1} (\partial_ta_{j,\alpha})(r,x)dr \Big|\\
				&\leq \P^{m-j} \Big|\int_{\tau+\varepsilon_2}^{\tau+\varepsilon_1}\frac{dr}{r^q}\Big|\\
				&\leq \P^{m-j}|E_q(\tau,\varepsilon_1,\varepsilon_2)|,
			\end{align*}
		\end{linenomath*}
		where
		\begin{linenomath*}
			\[
			E_q(\tau,\varepsilon_1,\varepsilon_2) =
			\left\{
			\begin{array}{ll}
				\log \Bigg(1+\frac{\varepsilon_1-\varepsilon_2}{\tau+\varepsilon_2}\Bigg)  & \mbox{if } q = 1 \\
				\frac{1}{q-1} \frac{(\tau+\varepsilon_1)^{q-1}-(\tau+\varepsilon_2)^{q-1}}{((\tau+\varepsilon_2)(\tau+\varepsilon_1))^{q-1}} & \mbox{if } q \in \big(1,\frac{3}{2}\big).
			\end{array}
			\right.
			\]
		\end{linenomath*}
		Note that $E_q(\tau,\varepsilon,\varepsilon)=0$.
		Then right-hand side of the inequality in (\ref{cone3}) is dominated by
		\begin{linenomath*}
			\begin{equation*}\label{cone4}
				C \int_{0}^{t} |E_q(\tau,\varepsilon_1,\varepsilon_2)|\sum_{j=0}^{m-1} \Vert (\partial_t^j v_{\varepsilon_2})(\tau,\cdot)\Vert_{\Phi,k;s+(m-1-j)e,\Lambda(\tau),\sigma}\;d\tau,
			\end{equation*}
		\end{linenomath*}
		where $C$ is independent of $\varepsilon$. By definition, $E_q$ is $L_1$-integrable in $\tau$.
		
		The sequence $v_{\varepsilon_k}$, $k=1,2,\dots$ corresponding to the sequence $\varepsilon_k \to 0$ is in the space
		\begin{linenomath*}
			\[
			\bigcap\limits_{j=0}^{m-1}C^{m-1-j}\Big([0,T^*];H^{s+(m-2-j)e,\Lambda(t),\sigma}_{\Phi,k}\Big), \quad T^*>0
			\]
		\end{linenomath*}
		and $u=\lim\limits_{k\to\infty}v_{\varepsilon_k}$ in the above space and hence, in $\mathcal{D}'(K(x^0,t^0))$. In particular,
		\begin{linenomath*}
			\[
			\la u,\varphi\ra = \lim\limits_{k\to\infty} \la v_{\varepsilon_k}, \varphi \ra =0,\; \forall \varphi \in \mathcal{D}(K(x^0,t^0)).
			\]
		\end{linenomath*}
		gives (\ref{cone2}) and completes the theorem. 
	\end{proof}
	
	\section*{Acknowledgements}
	The first author is funded by the University Grants Commission, Government of India, under its JRF and SRF schemes. The authors thank the anonymous referee for carefully reading the paper. We especially acknowledge the referee's point about the invertability of operators $I+K$ and $I+R$ (which appear in diagonalization procedure and change of variables, respectively) and for suggesting the use of calculus with $\japxin$, $k$ large, instead of $\japxi$.

	\addcontentsline{toc}{section}{ Appendix I: Lattice Structure for the Class of Metrics $(g_{\Phi,k})$}
			\section*{Appendix I: Lattice Structure for the Class of Metrics $(g_{\Phi,k})$}
	In this section, we define a lattice structure for the class of metrics of the form $g_{\Phi,k}$ defined as in (\ref{om}).
	Let $\mathcal{A}$ be a class of positive continuous functions that satisfy (\ref{sl})-(\ref{scale}). We endow $\mathcal{A}$ with a lattice structure by defining a partial order $\lesssim$ as $\Phi_1 \lesssim \Phi_2$ if for some $c>0$, $\Phi_1(x) \leq c\Phi_2(x), \forall x \in \R^n$ for $\Phi_1,\Phi_2 \in \mathcal{A}$. We say that $\Phi_1$ is equivalent to $\Phi_2$, denote it by $\Phi_1 \sim \Phi_2$, when the ratio $\Phi_1(x)/\Phi_2(x)$ is bounded both from above and below. Here the join($\vee$) and meet($\wedge$) operations are defined as $\Phi_1 \vee \Phi_2(x) \sim \max\{\Phi_1(x), \Phi_2(x)\}$ and $\Phi_1 \wedge \Phi_2(x) \sim \min\{\Phi_1(x), \Phi_2(x)\}$ respectively.
	
	The slowly varying nature of the function $\Phi \in \mathcal{A}$ allows us to introduce some partition of unity related to the corresponding metric $g_{\Phi,k}$, see  \cite[Theorem 2.2.7]{Lern}. This allows us to find a smooth $\tilde{\Phi}$ such that  $C^{-1}\Phi(x)\leq \tilde{\Phi}(x) \leq C \Phi(x)$ where $C$ is as in (\ref{sv}). The slowly varying nature of $\Phi$ dictates the precision of approximation $\tilde{\Phi}$ to $\Phi$. It is important to note that, in general, $\tilde{\Phi}$ does not belong to the class $\mathcal{A}$ but $C\tilde{\Phi} \in \mathcal{A}$ and $\Phi \sim C\tilde{\Phi}$. One consequence of this equivalence is that we can assume without loss of generality that the elements of $\mathcal{A}$ are smooth functions. 
	
	Let $\mathcal{B}$ be a class of $2n \times 2n$ symmetric positive definite matrices $Q_{\Phi}$ of the form
	\begin{linenomath*}
		\[
		Q_{\Phi}=\left(\begin{array}{@{}c|c@{}}
			\Phi(x)^{-2}I_n
			& \bigzero \\
			\hline
			\bigzero &
			\japxin^{-2}I_n
		\end{array}\right).
		\]
	\end{linenomath*}
	where $\Phi \in \mathcal{A}$ and $I_n$ is the identity matrix of size $n$. The quadratic form associated with $Q_{\Phi}$ is given by
	\begin{linenomath*}
		\[
		Q_{\Phi}(Y) := \la Q_{\Phi} Y,Y\ra = \frac{|y|^2}{\Phi(x)^2}+\frac{|\eta|^2}{\japxin^2}, \quad Y=(y,\eta) \in \R^{2n}.
		\]
	\end{linenomath*}
	Note that $\mathcal{B}$ has a natural lattice structure with a partial order relation defined by $Q_{\Phi_1} \lesssim Q_{\Phi_2}$ if for some $c>0$, $Q_{\Phi_1}(Y) \leq c Q_{\Phi_2}(Y)$ for all $Y\in \R^{2n}$. In similar fashion we can define the operations $\gtrsim$ and $\sim$. In this case the join and meet operations can be defined as $(Q_{\Phi_1} \vee Q_{\Phi_2})(Y) := Q_{\Phi_1\wedge\Phi_2}(Y)$ and $(Q_{\Phi_1} \wedge Q_{\Phi_2})(Y) := Q_{\Phi_1\vee\Phi_2}(Y)$.
	
	The natural map $T :\mathcal{A} \to \mathcal{B}$ defined by $ T(\Phi)=Q_\Phi$ defines an order-reversing isomorphism from $\mathcal{A}$ onto $\mathcal{B}$, i.e., $\Phi_1 \lesssim \Phi_2$ implies $T(\Phi_2) \lesssim T(\Phi_1)$ and the inverse is given by $T^{-1}(Q)=Q(y,0)^{-1}|y|^2$ for $0 \neq y\in\R^n$.
	We see that $T(\Phi_1 \wedge  \Phi_2) \sim T(\Phi_1)\vee  T(\Phi_2)$ and $T(\Phi_1 \vee \Phi_2) \sim T(\Phi_1) \wedge  T(\Phi_2)$.
	
	We can associate a Riemannian metric $g_{\Phi,k}$ to every $Q_\Phi \in \mathcal{B}$ as
	\[
	g_{\Phi,k} = \P^{-2}|dx|^2 + \japxin^{-2}|d\xi|^2, \quad (x,\xi)\in \R^{2n},
	\]
	and $g_{\Phi,k}(Y)=Q_\Phi(Y)$ for every $Y=(y,\eta)\in \R^{2n}$. Note that when $\Phi(x)=1$ and $k=1$, $g_{\Phi,k}$ corresponds to the metric $|dx|^2+\japxi^{-2}|d\xi|^2$ of the classical setting while the case $\Phi(x)=\japx$ and $k=1$ corresponds to the metric $\japx^{-2} |dx|^2 + \japxi^{-2}|d\xi|^2$ used in SG calculus \cite{Cordes}.
	
	Let $Q_{\Phi_1}$ and $Q_{\Phi_2}$ be two quadratic forms in $\mathcal{B}$. 
	\begin{lem}
		If $Q_{\Phi_2}(Y) \lesssim Q_{\Phi_1}(Y), Y \in \R^{2n}$, then the corresponding Riemannian metrics $g_{\Phi_1,k}$ and $g_{\Phi_2,k}$ are topologically equivalent.
	\end{lem}
	\begin{proof}
		We consider the identity map $I:(\R^{2n},g_{\Phi_1,k}) \to (\R^{2n},g_{\Phi_2,k})$. Since for any $Y_1,Y_2 \in \R^{2n}$ $g_{\Phi_2,k}(Y_2-Y_1) \leq C_1 g_{\Phi_1,k}(Y_2-Y_1)$ for some $C_1>0$, we see that $I$ is continuous. $I$ is also an open map by the open mapping theorem. This implies that $g_{\Phi_1,k}(Y_2-Y_1) \leq C_2g_{\Phi_2,k}(Y_2-Y_1)$ for some $C_2>0$. Thus the metrics are equivalent. 
	\end{proof}

	\addcontentsline{toc}{section}{Appendix II: Calculus for the Class $OPAG^{m_1,m_2}_{\Phi,k,\sigma;\sigma,\sigma}$}
	\section*{Appendix II: Calculus for the Class $OPAG^{m_1,m_2}_{\Phi,k,\sigma;\sigma,\sigma}$}

We now discuss in detail the symbol calculus for the class $AG^{m_1,m_2}_{\Phi,k,\sigma;\sigma,\sigma}$. The arguments used here are similar to the ones in \cite[Section 6.3]{nicRodi}. To start with, we prove that $P \in OPAG^{m_1,m_2}_{\Phi,k,\sigma;\sigma,\sigma}$ is continuous on $S_{\theta}^{\theta}\left(\R^{n}\right),\theta \geq \sigma$. Recall that $S_{\theta}^{\theta}\left(\R^{n}\right) \hookrightarrow \mathcal{M}^{2\theta}_{\Phi,k}(\R^n)$ and hence ${\mathcal{M}^{2\theta}_{\Phi,k}}'(\R^n) \hookrightarrow {S_{\theta}^{\theta}}'(\R^{n})$.
\begin{thm}
	Let $p \in AG^{m_1,m_2}_{\Phi,k,\sigma;\sigma,\sigma}$ and let $\theta \geq \sigma$. Then, the operator $P$ is a linear and continuous operator from $S_{\theta}^{\theta}\left(\R^{n}\right)$ to $S_{\theta}^{\theta}\left(\R^{n}\right)$ and it extends to a linear and continuous map
	from ${S_{\theta}^{\theta}}'(\R^{n})$ to ${S_{\theta}^{\theta}}'(\R^{n})$.
\end{thm}	
\begin{proof}[\bfseries{Proof}]
	Let $u \in S_{\theta}^{\theta}\left(\mathbb{R}^{d}\right)$. Since $\mathcal{F}\left(S_{\theta}^{\theta}\left(\mathbb{R}^{n}\right)\right)=S_{\theta}^{\theta}\left(\mathbb{R}^{n}\right),$ we consider $u$ in a
	bounded subset $F$ of the Banach space defined by the norm
	\begin{linenomath*}
		\[
		\sup _{\beta} \sup _{\xi \in \mathbb{R}^{n}} A^{-|\beta|}(\beta !)^{\theta} e^{a|\xi|^{1 / \theta}}\left|\partial^{\beta} \widehat{u}(\xi)\right|
		\]
	\end{linenomath*}
	for some $A>0, a>0$. It is sufficient to show that there exist positive constants $A_{1}, B_{1}, C_{1}$ such that, for every $\alpha, \beta \in \mathbb{N}^{n}$
	\begin{linenomath*}
		\begin{equation}\label{eq4}
			\sup _{x \in \mathbb{R}^{n}}\left|x^{\alpha} D_{x}^{\beta} P u(x)\right| \leq C_{1} A_{1}^{|\alpha|} B_{1}^{|\beta|}(\alpha ! \beta !)^{\theta}
		\end{equation}
	\end{linenomath*}
	for all $u \in F,$ with $A_{1}, B_{1}, C_{1}$ independent of $u \in F .$ We have, for every $N \in \mathbb{N}$
	\begin{linenomath*}
		\[
		\begin{aligned}
			x^{\alpha} D_{x}^{\beta} P u(x) &=x^{\alpha} \sum_{\beta^{\prime} \leq \beta}
			\begin{pmatrix}
				\beta \\
				\beta^{\prime}
			\end{pmatrix} \int e^{i x \xi} \xi^{\beta^{\prime}} D_{x}^{\beta-\beta^{\prime}} p(x, \xi) \widehat{u}(\xi) \textit{\dj} \xi \\
			&=x^{\alpha}\la x\ra^{-2 N} \sum_{\beta^{\prime} \leq \beta}\begin{pmatrix}
				\beta \\
				\beta^{\prime}
			\end{pmatrix} \int e^{i x \xi}\left(1-\Delta_{\xi}\right)^{N}\left[\xi^{\beta^{\prime}} D_{x}^{\beta-\beta^{\prime}} p(x, \xi) \widehat{u}(\xi)\right] \textit{\dj} \xi
		\end{aligned}
		\]
	\end{linenomath*}
	Since $\P \leq \japx$, we easily obtain the estimate:
	\begin{linenomath*}
		\[
		\begin{aligned}
			\left|x^{\alpha} D_{x}^{\beta} P u(x)\right| \leq C_{0} B_{0}^{|\beta|+2 N}(2 N !)^{\theta}\japx^{|\alpha|+n-2\gamma N} \\
			\times \sum_{\beta^{\prime} \leq \beta}\begin{pmatrix}
				\beta \\
				\beta^{\prime}
			\end{pmatrix}\left(\beta^{\prime} !\right)^{\theta}\left(\beta-\beta^{\prime} !\right)^{\nu} \int\japxin^{m} e^{-a|\xi|^{\frac{1}{\theta}}} \textit{\dj} \xi
		\end{aligned}
		\]
	\end{linenomath*}
	for some $B_{0}, C_{0}$ independent of $u \in F .$ Choosing $N \geq (|\alpha|+n)/\gamma$, we obtain that there exist $A_{1}, B_{1}, C_{1}>0$ such that $(\ref{eq4})$ holds for all $u \in F$. Next, observe that, for $u, v \in S_{\theta}^{\theta}(\mathbb{R}^{n})$
	\begin{linenomath*}
		\[
		\int P u(x) v(x) d x=\int \widehat{u}(\xi) p_{v}(\xi) \textit{\dj} \xi
		\]
	\end{linenomath*}
	where
	\begin{linenomath*}
		\[
		p_{v}(x, \xi)=\int e^{i x \xi} p(x, \xi) v(x) d x
		\]
	\end{linenomath*}
	By arguing as before, the map $v \mapsto p_{v}$ is linear and continuous from $S_{\theta}^{\theta}(\R^{n})$ to itself. Then, we can define, for $u \in {S_{\theta}^{\theta}}'(\R^{n})$
	\begin{linenomath*}
		\[
		P u(v)=\widehat{u}\left(p_{v}\right), \quad v \in {S_{\theta}^{\theta}}'(\R^{n})
		\]
	\end{linenomath*}
	This map extends $P$ and is linear and continuous from ${S_{\theta}^{\theta}}'(\R^{n})$ to itself and it.
\end{proof}
We can associate to $P$ a kernel $K_P \in {S_{\theta}^{\theta}}'(\R^{n})$, given by
\begin{linenomath*}
	\begin{equation}\label{ker}
		K_P(x,y) = \int e^{i(x-y)\xi}p(x,\xi) \textit{\dj} \xi.
	\end{equation}	
\end{linenomath*}
We now prove the following result showing the regularity of the kernel (\ref{ker}).
\begin{thm}
	Let $p \in AG^{m_1,m_2}_{\Phi,k,\sigma;\sigma,\sigma}$. For $M>0$, define the region:
	\begin{linenomath*}
		\[
		\Omega_M = \{(x,y) \in \R^{2n}: |x-y|>M\P\}.
		\]
	\end{linenomath*}
	Then the kernel $K_P$ defined by (\ref{ker}) is in $C^\infty(\Omega_M)$ and there exist positive constants $C,a$ depending on $M$ such that
	\begin{linenomath*}
		\[
		|D_x^\beta D_y^\alpha K_P(x,y)| \leq C^{|\alpha|+|\beta|+1} (\beta!\alpha!)^\theta \exp\Big\{ -a (\P\Phi(y))^{\frac{1}{\theta}} \Big\}
		\]
	\end{linenomath*}
	for all $(x,y) \in \bar{\Omega}_M$ and $\beta,\alpha \in \mathbb{Z}_+^n$.
\end{thm}
\begin{proof}[\bfseries{Proof}]
	As in \cite[Lemma 6.3.4]{nicRodi}, for any given $R>1$, we can find a sequence $\psi_N(\xi) \in C^\infty_0(\R^n), N=0,1,2,\dots$, such that $\sum_{N=0}^{\infty}\psi_n=1$,
	\begin{linenomath*}
		\[
		\text{supp}\psi_0 \subset \{\xi:\japxin \leq 3R\},
		\]
	\end{linenomath*}
	\begin{linenomath*}
		\[
		\text{supp}\psi_N \subset \{\xi: 2RN^\theta \japxin \leq 3R(N+1)^\theta\},  \quad \text{and}
		\]
	\end{linenomath*}
	\begin{linenomath*}
		\[
		|D_\xi^\alpha \psi_N(\xi)| \leq C^{|\alpha|+1}(\alpha!)^\theta \Big(R \sup\{N^\theta,1\}\Big)^{-|\alpha|},
		\]
	\end{linenomath*}
	for $ N=1,2,\dots,$ and $\alpha \in \mathbb{Z}_+^n$. Using this partition of unity, we can decompose $K_P$ as 
	\begin{linenomath*}
		\[
		K_P = \sum_{N=0}^{\infty}K_N,
		\]
	\end{linenomath*}
	where 
	\begin{linenomath*}
		\[
		K_N(x,y) = \int e^{i(x-y)\xi}p(x,\xi)\psi_N(\xi) \textit{\dj} \xi.
		\]
	\end{linenomath*}
	Let $M>0$ and $(x,y) \in \bar{\Omega}_M$. Let $r \in \{1,2,\dots,n\}$ such that $|x_r-y_r| \geq \frac{M}{n} \P$. Then, for every $\alpha, \beta \in \mathbb{Z}^n$,
	\begin{linenomath*}
		\begin{equation*}
			D_x^\alpha D_y^\beta K_N(x,y) = (-1)^{|\beta|} \sum_{\delta \leq \alpha} \begin{pmatrix*}\alpha \\ \delta\end{pmatrix*} \int e^{i(x-y)\xi} \xi^{\beta+\delta}\psi_N(\xi)D_x^{\alpha-\delta}p(x,\xi) \textit{\dj} \xi. 	\end{equation*}
	\end{linenomath*}
	Given $\lambda>0$, we consider the operator
	\begin{linenomath*}
		\[
		L = \frac{1}{M_{2\theta,\lambda}(x-y)} \sum_{j=0}^{\infty} \frac{\lambda^j}{(j!)^{2\theta}}(1-\Delta_\xi)^j, \quad \text{where}
		\]
	\end{linenomath*}
	\begin{linenomath*}
		\[
		M_{2\theta,\lambda}(x-y) = \sum_{j=0}^{\infty} \frac{\lambda^j}{(j!)^{2\theta}}\la x-y \ra^j.
		\]
	\end{linenomath*}
	Since $L e^{i(x-y)\xi} =e^{i(x-y)\xi}$, we can integrate by parts obtaining that
	\begin{linenomath*}
		\begin{align*}
			D_x^\alpha D_y^\beta K_P(x,y) &= \frac{(-1)^{|\beta|}}{M_{2\theta,\lambda}(x-y)} \sum_{\delta \leq \alpha} \begin{pmatrix*}\alpha\\ \delta\end{pmatrix*} \sum_{j=0}^{\infty} \frac{\lambda^j}{(j!)^{2\theta}} (1-\Delta_\xi)^j \\
			& \qquad \times \int e^{i(x-y)\xi} \xi^{\beta+\delta}\psi_N(\xi)D_x^{\alpha-\delta}p(x,\xi) \textit{\dj} \xi\\
			& = (-1)^{|\beta|+l}\frac{(x_r-y_r)^{-l}}{M_{2\theta,\lambda}(x-y)} \sum_{\delta \leq \alpha} \begin{pmatrix*}\alpha\\ \delta\end{pmatrix*} \sum_{j=0}^{\infty} \frac{\lambda^j}{(j!)^{2\theta}} \\
			& \qquad \times \int e^{i(x-y)\xi} D_{\xi_r}^l (1-\Delta_\xi)^j[ \xi^{\beta+\delta}\psi_N(\xi)D_x^{\alpha-\delta}p(x,\xi)] \textit{\dj} \xi
		\end{align*}
	\end{linenomath*}
	where $l \in \mathbb{Z}_+$ is choosen appropriately. Let
	\begin{linenomath*}
		\[
		F_{rljN\alpha\beta\delta} = D_{\xi_r}^l (1-\Delta_\xi)^j[ \xi^{\beta+\delta}\psi_N(\xi)D_x^{\alpha-\delta}p(x,\xi)].
		\]
	\end{linenomath*}
	We denote by $e_r$ the $r$-th vector of the canonical basis of $\R^n$ and $\la \beta,e_r \ra = \beta_r$, $\la \delta,e_r \ra = \delta_r$. We see that
	\begin{linenomath*}
		\begin{align*}
			F_{rljN\alpha\beta\delta} = &\sum_{\substack{l_1+l_2+l_3=l \\ l_1 \leq \beta_r+\delta_r}}(-i)^{l_1} \frac{l!}{l_1!l_2!l_3!} \frac{(\beta_r+\delta_r)!}{(\beta_r+\delta_r-l_1)!} \\
			&\times (1-\Delta_\xi)^j[ \xi^{\beta+\delta-l_1e_r} D_{\xi_r}^ {l_2}\psi_N(\xi) D_{\xi_r}^{l_3}D_x^{\alpha-\delta}p(x,\xi)].
		\end{align*}
	\end{linenomath*}
	Hence
	\begin{linenomath*}
		\begin{align*}
			|F_{rljN\alpha\beta\delta}| = &\sum_{\substack{l_1+l_2+l_3=l \\ l_1 \leq \beta_r+\delta_r}}(-i)^{l_1} \frac{l!}{l_1!l_2!l_3!} \frac{(\beta_r+\delta_r)!}{(\beta_r+\delta_r-l_1)!} C_1^{|\alpha-\delta|+l_2+l_3+1} \\
			&\times (N_2!)^\theta (N_3!)^\sigma[(\alpha-\delta)!]^\sigma C_2^j(j!)^{2\theta} \bigg(\frac{1}{RN^\theta}\bigg)^{l_2} \\
			&\times \japxin^{m_1- \gamma l_3 + |\alpha-\delta|/\sigma+|\beta|+|\delta|-l_1} \P^{m_2-\gamma|\alpha-\delta|+(l_3+j)/\sigma}.
		\end{align*}
	\end{linenomath*}
	Note that on the support of $\psi_n$, $RN^\theta \japxin \leq 3R(N+1)^\theta$. Since $\theta \geq \sigma$, it follows that
	\begin{linenomath*}
		\[
		|F_{rljN\alpha\beta\delta}| \leq C_1^{|\alpha|+|\beta|+1}(\alpha!\beta!)^\theta C_2^j (j!)^{2\theta} \bigg(\frac{C_3}{R}\bigg)^{\gamma l}\P^{m_2-\gamma|\alpha-\delta|+(l_3+j)/\sigma},
		\] 
	\end{linenomath*}
	with $C_3$ independent of $R$. We now choose $l$ such that $\gamma l>N$. We observe that for every $c>1$ there exist positive constants $\varepsilon,c'$ such that, for $\tau>0$,
	\begin{linenomath*}
		\begin{equation}\label{ineq}
			\varepsilon \exp[c'\tau] \leq \sum_{j=0}^{\infty} \bigg(\frac{\tau^j}{j!}\bigg)^c.
		\end{equation}
	\end{linenomath*}
	Setting $c=\theta,\tau = \lambda^{\frac{1}{\theta}}\la x-y\ra^{\frac{2}{\theta}}$, we have that
	\begin{linenomath*}
		\[
		|M_{2\theta,\lambda}(x-y)| \geq \varepsilon \exp\{c'\lambda^{\frac{1}{\theta}} | x-y|^{\frac{2}{\theta}}\}.
		\]
	\end{linenomath*}
	Observe that $|x-y|^2 \geq c'' \japx \la y \ra$. From these estimates, choosing $\lambda<C_2^{-1}$ and $R$ sufficiently large, we see that 
	\begin{linenomath*}
		\[
		|D_x^\alpha D_y^\beta K_P(x,y)| \leq C_1^{|\alpha|+|\beta|+1} (\alpha!\beta!)^\theta \bigg(\frac{C_4}{R}\bigg)^{N} \exp\Big\{-\frac{c'}{2}\lambda^{\frac{1}{\theta}}(\P\Phi(y))^{\frac{1}{\theta}}\Big\},
		\]
	\end{linenomath*}
	where $C_4$ is independent of $R$.
\end{proof}

\begin{defn}
	A linear continuous operator from $\mathcal{S}^\theta_\theta(\R^n)$  to $\mathcal{S}^\theta_\theta(\R^n)$ is said to be $\theta$-regularizing if it extends to a linear continuous map from ${\mathcal{S}^\theta_\theta}'(\R^n)$ to $\mathcal{S}^\theta_\theta(\R^n)$.
\end{defn}	

\begin{defn}
	A linear continuous operator from $\mathcal{M}^\theta_{\Phi,k}(\R^n)$  to $\mathcal{M}^\theta_{\Phi,k}(\R^n)$ is said to be $(\Phi,\theta)$-regularizing if it extends to a linear continuous map from ${\mathcal{M}^\theta_{\Phi,k}}'(\R^n)$ to $\mathcal{M}^\theta_{\Phi,k}(\R^n)$.
\end{defn}

For $t>1$, we set
\begin{linenomath*}
	\[
	Q_t=\{(x,\xi) \in \R^{2n}: \P\japxin<t\},
	\]
\end{linenomath*}
and 
\begin{linenomath*}
	\[
	Q_t^e= \R^{2n} \setminus Q_t.
	\]
\end{linenomath*}
\begin{defn} \label{FS}
	We denote by $FAG^{m_1,m_2}_{\Phi,k,\sigma;\sigma,\sigma}$ the space of all formal sums $\sum_{j \geq 0}p_j(x,\xi)$ such that $p_j(x,\xi) \in C^\infty(\R^{2n})$
	such that for all $ j\geq 0$ and there exists $B,C>0$ such that
	\begin{linenomath*}
		\begin{align*}
			\sup_{j \geq 0}\sup_{\alpha,\beta \in \mathbb{Z}_+^n} &\sup_{(x,\xi) \in  Q_{Bj^{2\sigma-2}}^e} C^{-|\alpha|-|\beta|-2j}(\alpha!)^{-\sigma} (\beta!)^{-\sigma} (j!)^{-2\sigma+2} \\
			&\japxin^{-m_1+j+\gamma|\alpha|-|\beta|/\sigma} \P^{-m_2+j+\gamma|\beta|-|\alpha|/\sigma} |D_\xi^\alpha \partial_x^\beta p_j(x,\xi)| < + \infty.
		\end{align*}
	\end{linenomath*}
\end{defn}
Every symbol $p \in AG^{m_1,m_2}_{\Phi,k,\sigma;\sigma,\sigma}$ can be identified with an element of $FAG^{m_1,m_2}_{\Phi,k,\sigma;\sigma,\sigma},$ by setting $p_0=p$ and $p_j=0$ for $j \geq 1$.
\begin{defn}
	Two sums $\sum_{j \geq 0}p_j$ , $\sum_{j \geq 0}p_j'$ are said to be equivalent if there exist constants $A,B>0$ such that 
	\begin{linenomath*}
		\begin{align*}
			\sup_{N\in \mathbb{Z_+}} &\sup_{\alpha,\beta \in \mathbb{N}^n} \sup_{(x,\xi) \in  Q_{BN^{2\sigma-2}}^e} C^{-|\alpha|-|\beta|-2N}(\alpha!)^{-\sigma} (\beta!)^{-\sigma} (N!)^{-2\sigma+2} \\
			&\japxin^{-m_1+N+\gamma|\alpha|-|\beta|/\sigma} \P^{-m_2+N+\gamma|\beta|-|\alpha|/\sigma} |D_\xi^\alpha \partial_x^\beta \sum_{j<N} (p_j-p_j')| < + \infty,
		\end{align*}
	\end{linenomath*}
	and we write $\sum_{j \geq 0}p_j \sim \sum_{j \geq 0}p_j'$.
\end{defn}

\begin{prop}\label{sum}
	Given $\sum_{j \geq 0}p_j \in FAG^{m_1,m_2}_{\Phi,k,\sigma;\sigma,\sigma}$, there exists a symbol $p \in AG^{m_1,m_2}_{\Phi,k,\sigma;\sigma,\sigma}$ such that
	\begin{linenomath*}
		\[
		p \sim \sum_{j \geq 0}p_j \quad \text{ in } \: FAG^{m_1,m_2}_{\Phi,k,\sigma;\sigma,\sigma}.
		\]
	\end{linenomath*}
\end{prop}
\begin{proof}[\bfseries{Proof}]
	To construct the symbol $p$, we consider the excision function $\varphi \in C^\infty(\R^{2n})$ such that $0 \leq \varphi \leq 1$ and $\varphi(x,\xi) = 0$ if $(x,\xi) \in Q_2$, $\varphi(x,\xi) = 1$ if $(x,\xi) \in Q_3^e$ and 	
	\begin{linenomath*}
		\begin{equation} \label{exci}
			\sup _{(x, \xi) \in \mathbb{R}^{2 n}}\left|D_{\xi}^{\alpha} D_{x}^{\beta} \varphi(x, \xi)\right| \leq C^{|\alpha|+|\beta|+1}(\alpha !\beta !)^{\sigma}	.	 
		\end{equation}
	\end{linenomath*}
	We define for $R>0$
	\begin{linenomath*}
		$$\begin{array}{c}
			\varphi_{0}(x, \xi) = \varphi(x, \xi) \quad \text { on } \mathbb{R}^{2 d} \\
			\varphi_{j}(x, \xi)=\varphi\left(\frac{x}{R {j}^{\sigma-1}}, \frac{\xi}{R {j}^{\sigma-1}}\right), \quad j \geq 1.
		\end{array}$$
	\end{linenomath*}
	For sufficiently large $R$, we will prove that 
	\begin{linenomath*}
		\[
		p(x, \xi)=\sum_{j \geq 0} \varphi_{j}(x, \xi) p_{j}(x, \xi),
		\]
	\end{linenomath*}
	is in $AG^{m_1,m_2}_{\Phi,k,\sigma;\sigma,\sigma}$ and $	p \sim \sum_{j \geq 0}p_j $ in $ FAG^{m_1,m_2}_{\Phi,k,\sigma;\sigma,\sigma}$. From Definition \ref{FS} ,we have
	\begin{linenomath*}
		\begin{align*}
			|D_{\xi}^{\alpha} D_{x}^{\beta} p(x, \xi)| &= \big|\sum_{j \geq 0} \sum\limits_{\substack{\alpha' \leq \alpha \\ \beta' \leq \beta}} \begin{pmatrix} \alpha \\ \alpha' \end{pmatrix}
			\begin{pmatrix}	\beta \\  \beta' \end{pmatrix} 
			D_{\xi}^{\alpha-\alpha'} D_{x}^{\beta-\beta'} p_{j}(x, \xi) D_{\xi}^{\alpha'} D_{x}^{\beta'} \varphi_{j}(x, \xi)\big|\\
			&\leq C^{|\alpha|+|\beta|+1} \alpha ! \beta !\japxin^{m_1-\gamma|\alpha|+|\beta|/\sigma} \Phi^{m_2-\gamma|\beta|+|\alpha|/\sigma} \sum_{j \geq 0} H_{j \alpha \beta}(x, \xi)
		\end{align*}
	\end{linenomath*}
	where
	\begin{linenomath*}
		\begin{align*}
			H_{j \alpha \beta}(x, \xi) = \sum_{\alpha' \leq \alpha \atop \beta' \leq \beta}  &\frac{\big((\alpha-\alpha') ! (\beta-\beta') !\big)^{\sigma-1}}{\alpha' ! \beta' !} C^{2 j-|\alpha'|-|\beta'|}(j !)^{2\sigma-2} \\
			& \times\japxin^{\gamma|\alpha'|-j}\P^{\gamma|\beta'|-j}\left|D_{\xi}^{\alpha'} D_{x}^{\beta'} \varphi_{j}(x, \xi)\right|.
		\end{align*}
	\end{linenomath*}
	From (\ref{exci}), we have
	\begin{linenomath*}
		\[
		H_{j \alpha \beta}(x, \xi) \leq C^{|\alpha|+|\beta|+1}(\alpha !)^{\sigma-1}(\beta !)^{\sigma-1}\left(\frac{C_{1}}{R^2}\right)^{j}
		\]
	\end{linenomath*}
	where $C_1>0$ is independent of $R$. By choosing $R$ sufficiently large, we obtain the required result. We observe that
	\begin{linenomath*}
		\[
		p(x, \xi)-\sum_{j<N} p_{j}(x, \xi)=\sum_{j \geq N} p_{j}(x, \xi) \varphi_{j}(x, \xi),
		\]
	\end{linenomath*}
	for $(x,\xi) \in Q^e_{3RN^{2\sigma-1}}, N \in \mathbb{Z}_+$, which we can estimate as before.
\end{proof}
\begin{prop}\label{zero}
	Let $p \in AG^{0,0}_{\Phi,k,\sigma;\sigma,\sigma}$ and $\theta \geq 2 (\sigma-1)$. If $p \sim 0$ in $FAG^{0,0}_{\Phi,k,\sigma;\sigma,\sigma}$, then the operator $P$ is $(\Phi,\theta)$-regularizing.
\end{prop}
\begin{proof}[\bfseries{Proof}]
	We will show that the kernel
	\[K_{P}(x, y)=(2 \pi)^{-d} \int e^{i(x-y) \xi} p(x, \xi) d \xi\]
	is in $\mathcal{M}^\theta_{\Phi,k}(\R^{2n})$ impliying that $P$ is $(\Phi,\theta)$-regularizing.
	
	There exist $B,C>0$ such that for every $(x,\xi)\in \R^{2n}$:
	\begin{linenomath*}
		\begin{align*}
			\left|D_{\xi}^{\alpha} D_{x}^{\beta} p(x, \xi)\right| &\leq C_{1}^{|\alpha|+|\beta|+1}(\alpha !\beta !)^{\sigma}\japxin^{-|\alpha|}\P^{-|\beta|} \\
			& \qquad \times \inf _{0 \leq N \leq B_{1}(\japxin\P)^{\frac{1}{2\sigma-2}}} \frac{C^{2 N}(N !)^{2\sigma-2}}{\japxin^{N}\P^{N}}.
		\end{align*}
	\end{linenomath*}
	Using \cite[Lemma 6.3.10]{nicRodi}, we obtain
	\begin{linenomath*}
		\[
		\left|D_{\xi}^{\alpha} D_{x}^{\beta} p(x, \xi)\right| \leq C_{2}^{|\alpha|+|\beta|+1}(\alpha ! \beta !)^{\theta} \exp \left[-a\Big(\P\japxin\Big)^{\frac{1}{\theta}}\right]
		\]
	\end{linenomath*}
	for some $C_2,a>0$.
\end{proof}

We now examine the stability of the classes $OPAG^{m_1,m_2}_{\Phi,k,\sigma;\sigma,\sigma}$ under transposition, composition and construction of parametrices. Let $u\in {\mathcal{M}^{\theta'}_{\Phi,k}}(\R^n)$ and $v \in \mathcal{M}^\theta_{\Phi,k}(\R^n), \theta \geq 2(\sigma-1)$. To relate with Gelfand-Shilov spaces, one can even consider $u\in {\mathcal{S}^{\frac{\theta}{2}}_{\frac{\theta}{2}}}'(\R^n)$ and $v \in \mathcal{S}^{\frac{\theta}{2}}_{\frac{\theta}{2}}(\R^n)$.
\begin{prop}\label{trans}
	Let $P=p(x,D) \in AG^{m_1,m_2}_{\Phi,k,\sigma;\sigma,\sigma}$ and let $P^t$ be the transposed operator defined by
	\begin{linenomath*}
		\begin{equation}\label{tran}
			\la P^tu,v\ra =\la u,Pv \ra.
		\end{equation}
	\end{linenomath*}
	Then, $P^t =Q+R$, where $R$ is $(\Phi,\theta)$-regularizing and $Q=q(x,D)$ is in $AG^{m_1,m_2}_{\Phi,k,\sigma;\sigma,\sigma}$ with
	\begin{linenomath*}
		\[
		q(x,\xi) \sim \sum _{j \geq 0}\sum_{\alpha =j}\frac{1}{\alpha!} \partial_\xi^\alpha D_{x}^\alpha p(x,-\xi) \quad \text{ in } FAG^{m_1,m_2}_{\Phi,k,\sigma;\sigma,\sigma}.
		\]
	\end{linenomath*}
\end{prop}
\begin{thm}\label{comp}
	Let $P=p(x,D) \in AG^{m_1,m_2}_{\Phi,k,\sigma;\sigma,\sigma}$ ,$Q=q(x,D)\in AG^{m_1',m_2'}_{\Phi,k,\Phi,\sigma;\sigma,\sigma}$. Then $PQ=T+R$ where $R$ is $(\Phi,\theta)$-regularizing and $T=t(x,D)$ is in $AG^{m_1+m_1',m_2+m_2'}_{\Phi,k,\sigma;\sigma,\sigma}$ with
	\begin{linenomath*}
		\[
		t(x,\xi) \sim \sum _{j \geq 0}\sum_{\alpha =j}\frac{1}{\alpha!} \partial_\xi^\alpha p(x,\xi) D_{x}^\alpha q(x,\xi)
		\]
	\end{linenomath*}
	in $FAG^{m_1+m_1',m_2+m_2'}_{\Phi,k,\sigma;\sigma,\sigma}$.
\end{thm} 

To prove Proposition \ref{trans} and Theorem \ref{comp}, we introduce more general classes of symbols, called amplitudes. Let $(m_1,m_2,m_3) \in \R^3$.
\begin{defn}
	We denote by $\Pi_{\Phi,k,\sigma;\sigma}^{m_1,m_2,m_3}$ the Banach space of all symbols $a(x,y,\xi) \in C^\infty(\R^3)$ satisfying for some $C>0$ the following estimate 
	\begin{linenomath*}
		\begin{align*}
			\sup_{\alpha,\beta,\delta \in \mathbb{Z}_+^n} &\sup_{(x,y,\xi)\in \R^3} C^{-|\alpha|-|\beta|-|\delta|}(\alpha!\beta!\delta!)^{-\sigma} \japxin^{-m_1+\gamma|\alpha|-(|\beta|+|\delta|)/\sigma}\\
			& \P^{-m_2+\gamma|\beta|-|\alpha|/\sigma} \Phi( y )^{-m_3+\gamma|\delta|-|\alpha|/\sigma}|D_\xi^\alpha D_x^\beta D_y^\delta a(x,y,\xi)| < + \infty.
		\end{align*}
	\end{linenomath*}
\end{defn}
Given $a \in \Pi^{m_1,m_2,m_3}_{\Phi,k,\sigma;\sigma}$, we associate to $a$ the pseudodifferential operator defined by
\begin{linenomath*}
	\begin{equation}\label{ampl}
		Au(x) =\iint\limits_{\R^{2n}}e^{i(x-y)\cdot\xi}a(x,y,\xi){u}(y)dy\textit{\dj}\xi, \quad u \in \mathcal{M}^\theta_\Phi(\R^n)
	\end{equation}
\end{linenomath*}
\begin{thm}\label{amptosym}
	Let $A$ be an operator defined by an amplitude $a \in \Pi_{\Phi,k,\sigma;\sigma}^{m_1,m_2,m_3}$, $(m_1,m_2,m_3) \in \R^3$. Then we may write $A=P+R$, where $R$ is a $(\Phi,\theta)$-regularizing operator  and $P=p(x,D)$ is in $AG^{m_1,m_2+m_3}_{\Phi,k,\sigma;\sigma,\sigma}$ with $p \sim \sum_{j \geq 0}{p_j}$ where
	\begin{linenomath*}
		\[
		p_j(x,\xi) = \sum _{|\alpha|=j}\frac{1}{\alpha!} \partial_\xi^\alpha D_{y}^\alpha a(x,y,\xi)|_{y=x}.
		\]
	\end{linenomath*}
\end{thm}
The proof of this result uses the similar standard arguments available in the proof of Theorem $6.3.14$ in \cite{nicRodi}. For the sake of conciseness, we omit the details.

\begin{proof}[\bfseries{Proof of Theorem~\ref{trans}}]
	From (\ref{tran}), $P^{t}$ is defined as
	\begin{linenomath*}
		\[
		P^{t} u(x)=\int e^{i(x-y) \xi} p(y,-\xi) u(y) d y \textit{\dj} \xi, \quad u \in \mathcal{M}^{\theta}_{\Phi,k}\left(\mathbb{R}^{n}\right)
		\]
	\end{linenomath*}
	Observe that $ P^{t}$ is an operator of the form $(\ref{ampl})$ with amplitude $p(y,-\xi) .$ By Theorem $(\ref{amptosym})$, $P^{t}=Q+R$ where $R$ is $(\Phi,\theta)$-regularizing and $Q=q(x, D) \in AG^{m_1,m_2}_{\Phi,k,\sigma;\sigma,\sigma}$
	with
	\begin{linenomath*}
		\[
		q(x, \xi) \sim \sum_{j \geq 0} \sum_{|\alpha|=j}(\alpha !)^{-1} \partial_{\xi}^{\alpha} D_{x}^{\alpha} p(x,-\xi).
		\]
	\end{linenomath*}
\end{proof}

\begin{proof}[\bfseries{Proof of Theorem~\ref{comp}}]
	We can write $Q=(Q^{t})^{t} .$ Then, by Theorem (\ref{amptosym}) and Proposition (\ref{trans}), $Q=Q_{1}+R_{1},$ where $R_{1}$ is $(\Phi,\theta)$-regularizing and
	\begin{linenomath*}
		\begin{equation}\label{eq3}
			Q_{1} u(x)=\int e^{i(x-y) \xi} q_{1}(y, \xi) u(y) d y \textit{\dj} \xi
		\end{equation}
	\end{linenomath*}
	with $q_{1}(y, \xi) \in A G_{\Phi,k,\sigma;\sigma,\sigma}^{m_1^{\prime}, m_2^{\prime}}, q_{1}(y, \xi) \sim \sum_{\alpha}(\alpha !)^{-1} \partial_{\xi}^{\alpha} D_{y}^{\alpha} q(y,-\xi) .$ From $(\ref{eq3})$ it
	follows that
	\begin{linenomath*}
		\[
		\widehat{Q_1 u}(\xi)=\int e^{-i y \xi}q_1(y, \xi) u(y) d y, \quad u \in \mathcal{M}^{\theta}_{\Phi,k}(\mathbb{R}^{n})
		\]
	\end{linenomath*}
	from which we deduce that
	\begin{linenomath*}
		\[
		P Q u(x)=\int e^{i(x-y)\xi} p(x, \xi) q_{1}(y, \xi) u(y) d y \textit{\dj} \xi+P R_{1} u(x).
		\]
	\end{linenomath*}
	We observe that $p(x, \xi) q_{1}(y, \xi) \in \Pi_{\Phi,k,\sigma;\sigma}^{m_1+m_1^{\prime}, m_2, m_2^{\prime}}$. Applying Theorem (\ref{amptosym}), we obtain that
	\begin{linenomath*}
		\[
		P Q u(x)=T u(x)+R u(x)
		\]
	\end{linenomath*}
	where $R$ is $(\Phi,\theta)$-regularizing and $T=t(x, D) \in OPAG_{\Phi,k,\sigma;\sigma,\sigma}^{m_1+m_1^{\prime}, m_2+m_2^{\prime}}$ with
	\begin{linenomath*}
		\[
		t(x, \xi) \sim \sum_{j \geq 0} \sum_{|\alpha|=j}(\alpha !)^{-1} \partial_{\xi}^{\alpha} p(x, \xi) D_{x}^{\alpha} q(x, \xi)
		\]
	\end{linenomath*}
	in $FAG^{m_1+m_1^{\prime}, m_2+m_2^{\prime}}_{\Phi,k,\sigma;\sigma,\sigma}$.
\end{proof}

We now state the notion of ellipticity for elements of $OPAG_{\Phi,k,\sigma;\sigma,\sigma}^{m_1, m_2}$.

\begin{defn}
	A symbol $p \in AG_{\Phi,k,\sigma;\sigma,\sigma}^{m_1, m_2}$ is said to be $G_\Phi$-elliptic if there exist $B, C=C(k)>0$ such that
	\begin{linenomath*}
		\[
			|p(x, \xi)| \geq C \japxin^{m_1} \P^{m_2}, \quad \forall (x, \xi) \in Q_{B}^{e}.
		\]
	\end{linenomath*}
\end{defn}
\begin{thm}\label{ell}
	If $p \in AG_{\Phi,k,\sigma;\sigma,\sigma}^{m_1, m_2}$ is $G_\Phi$-elliptic and $P=p(x, D),$ then there erists $E \in AG_{\Phi,k,\sigma;\sigma,\sigma}^{-m_1, -m_2}$ such that $E P=I+R_{1}, P E=I+R_{2},$ where $R_{1}, R_{2}$ are
	$(\Phi,\theta)$-regularizing operators.
\end{thm}
The above theorem can be easily proved using Proposition \ref{sum} and Proposition \ref{comp}. For the sake of conciseness, we omit the proof and  refer the reader to Theorem $6.3.16$ in \cite{nicRodi} for the details.

As an immediate consequence of Theorem $\ref{ell}$ we obtain the following result of global regularity.

\begin{cor}
	Let $p \in AG_{\Phi,k,\sigma;\sigma,\sigma}^{m_1, m_2}$ be $G_\Phi$-elliptic and let $f \in \mathcal{M}_{\Phi,k}^{\theta}(\R^n)$. If $u \in {\mathcal{M}_{\Phi,k}^{\theta}}^\prime(\R^n)$ is a solution of the equation
	\begin{linenomath*}
		\[	
		P u=f,		
		\]
	\end{linenomath*}
	then $u \in \mathcal{M}_{\Phi,k}^{\theta}(\R^n)$.
\end{cor}	

This completes the calculus for the class $OPAG^{m_1,m_2}_{\Phi,k,\sigma;\sigma,\sigma}$. 

\addcontentsline{toc}{section}{Appendix III: Calculus for the Class $OPAG^{m_1,m_2}_{\Phi,k;\sigma}$}

	\section*{Appendix III: Calculus for the Class $OPAG^{m_1,m_2}_{\Phi,k;\sigma}$}
Using similar techniques as in the case of calculus of  $OPAG^{m_1,m_2}_{\Phi,k,\sigma;\sigma,\sigma}$ and referring to Appendix A in \cite{AscaCappi2}, one can also develop a calculus for the class $OPAG^{m_1,m_2}_{\Phi,k;\sigma}$. We start with defining the notion of formal sums in this context.
\begin{defn} \label{FS2}
	We denote by $FAG^{m_1,m_2}_{\Phi,k;\sigma}$ the space of all formal sums $\sum_{j \geq 0}p_j(x,\xi)$ such that $p_j(x,\xi) \in C^\infty(\R^{2n})$
	such that for all $ j\geq 0$ and following conditions hold
	\begin{enumerate}[label=\roman*)]
		\item There exist $B,C>0$ such that
		\begin{linenomath*}
			\begin{align}\label{fs1}
				\sup_{j \geq 0}\sup_{\alpha,\beta \in \mathbb{Z}_+^n} &\sup_{(x,\xi) \in  Q_{B(j+|\alpha|)^{\sigma}}^e} C^{-|\alpha|-|\beta|-2j}(\alpha!)^{-1} (\beta!j!)^{-\sigma} \nonumber\\
				&\japxin^{-m_1+j+|\alpha|} \P^{-m_2+j+|\beta|} |D_\xi^\alpha \partial_x^\beta p_j(x,\xi)| < + \infty.
			\end{align}
		\end{linenomath*}
		\item For every $B_0>0$, there exists $C>0$ such that
		\begin{linenomath*}
			\begin{align}\label{fs2}
				\sup_{j \geq 0}\sup_{\alpha,\beta \in \mathbb{Z}_+^n} &\sup_{(x,\xi) \in  Q_{B_0(j+|\alpha|)^{\sigma}}} C^{-|\alpha|-|\beta|-2j} (\alpha!\beta!j!)^{-\sigma} 
				|D_\xi^\alpha \partial_x^\beta p_j(x,\xi)| < + \infty.
			\end{align}
		\end{linenomath*}
	\end{enumerate}		
\end{defn}
Every symbol $p \in AG^{m_1,m_2}_{\Phi,k;\sigma}$ can be identified with an element of $FAG^{m_1,m_2}_{\Phi,k;\sigma},$ by setting $p_0=p$ and $p_j=0$ for $j \geq 1$.
\begin{defn}
	Two sums $\sum_{j \geq 0}p_j$ , $\sum_{j \geq 0}p_j'$ in $FAG^{m_1,m_2}_{\Phi,k;\sigma}$ are said to be equivalent if there exist constants $A,B>0$ such that 
	\begin{linenomath*}
		\begin{align*}
			\sup_{N\in \mathbb{Z_+}} &\sup_{\alpha,\beta \in \mathbb{N}^n} \sup_{(x,\xi) \in  Q_{B(N+|\alpha|)^{\sigma}}^e} C^{-|\alpha|-|\beta|-2N}(\alpha!)^{-1} (\beta!N!)^{-\sigma} \\
			&\japxin^{-m_1+N+|\alpha|} \P^{-m_2+N+|\beta|} |D_\xi^\alpha \partial_x^\beta \sum_{j<N} (p_j-p_j')| < + \infty,
		\end{align*}
	\end{linenomath*}
	and we write $\sum_{j \geq 0}p_j \sim \sum_{j \geq 0}p_j'$.
\end{defn}

\begin{prop}
	Given $\sum_{j \geq 0}p_j \in FAG^{m_1,m_2}_{\Phi,k;\sigma}$, there exists a symbol $p \in AG^{m_1,m_2}_{\Phi,k;\sigma}$ such that
	\begin{linenomath*}
		\[
		p \sim \sum_{j \geq 0}p_j \quad \text{ in } \: FAG^{m_1,m_2}_{\Phi,k;\sigma}.
		\]
	\end{linenomath*}
\end{prop}
\begin{proof}[\bfseries{Proof}]		
	Following \cite{AscaCappi2}, to construct the symbol $p$, we consider two types of cut-off functions. For a fixed $R>0$, we can find a sequence of functions $\psi_{j}(\xi), j=0,1,2, \ldots,$ such that $0 \leqslant \psi_{j}(\xi) \leqslant 1$ for all $\xi \in \R^{n},\psi_{j}(\xi)=1$ if $\japxin \leqslant 2 R \sup \left(j^{\frac{\sigma}{2}}, 1\right), \psi_{j}(\xi)=0$ if $\japxin \geqslant 4 R \sup \left(j^{\frac{\sigma}{2}}, 1\right)$
	and satisfying the following estimates:
	\begin{linenomath*}
		\[
		\left|\partial_{\xi}^{\alpha} \psi_{j}(\xi)\right| \leqslant C_{1}^{|\alpha|}\left(R \sup \left(j^{\sigma-1}, 1\right)\right)^{-|\alpha|} \quad \text { if }|\alpha| \leqslant 4 j,
		\]
	\end{linenomath*}
	and
	\begin{linenomath*}
		\[
		\left|\partial_{\xi}^{\alpha} \psi_{j}(\xi)\right| \leqslant C_{2}^{|\alpha|+1}(\alpha !)^{\sigma}\left(R \sup \left(j^{\sigma}, 1\right)\right)^{-|\alpha|} \quad \text { if }|\alpha|>4 j,
		\]
	\end{linenomath*}
	for some positive constants $C_{1}, C_{2}$ independent of $\alpha, R, j$.
	
	Similarly, we can choose a sequence of functions $\tilde{\psi}_{j}(x) \in G_{0}^{\sigma}(\R^{n}), j=0,1,2, \ldots,$ supported for $\P \leqslant 4 R \sup \left(j^{\frac{\sigma}{2}}, 1\right), \psi_{j}(x)=1$ for $\P \leqslant 2 R \sup \left(j^{\frac{\sigma}{2}}, 1\right)$ and
	\begin{linenomath*}
		\[
		|\partial_{x}^{\beta} \tilde{\psi}_{j}(x)| \leqslant C_{3}^{|\beta|+1}(\beta !)^{\sigma}\left(R \sup \left(j^{\sigma}, 1\right)\right)^{-|\beta|} \quad \text { for all } x \in \R^{n}, \beta \in \mathbb{Z}_{+}^{n}.
		\]
	\end{linenomath*}
	Let us now define $\varphi_{j}(x, \xi)=(1-\psi_{j}(\xi))(1-\tilde{\psi}_{j}(x)), j=0,1,2, \ldots .$ By the properties listed above, we deduce that the functions $\varphi_{j}$ are smooth on $\R^{2 n},$ supported in $Q_{4 R^2 \sup \left(j^{\sigma}, 1\right)}^{e}$ and $\varphi_{j}(x, \xi)=1$ in $Q_{16 R^2 \sup \left(j^{\sigma}, 1\right)}^{e} .$ Moreover
	\begin{linenomath*}
		\[
		\left|D_{\xi}^{\alpha} D_{x}^{\beta} \varphi_{0}(x, \xi)\right| \leqslant A\left(\frac{C}{R}\right)^{|\alpha|+|\beta|}
		\]
	\end{linenomath*}
	whereas, for $j \geqslant 1$ the functions $\varphi_{j}$ satisfy the following estimates:
	\begin{linenomath*}
		\begin{equation}\label{cut1}
			\left|D_{\xi}^{\alpha} D_{x}^{\beta} \varphi_{j}(x, \xi)\right| \leqslant A\left(\frac{C}{R}\right)^{|\alpha|+|\beta|}(\beta !)^{\sigma}\left(j^{\sigma-1}\right)^{-|\alpha|} j^{-\sigma|\beta|}
		\end{equation}
	\end{linenomath*}
	for $|\alpha| \leqslant 3 j, \beta \in \mathbb{Z}_{+}^{n}$ and
	\begin{linenomath*}
		\begin{equation}\label{cut2}
			\left|D_{\xi}^{\alpha} D_{x}^{\beta} \varphi_{j}(x, \xi)\right| \leqslant A\left(\frac{C}{R}\right)^{|\alpha|+|\beta|}(\alpha ! \beta !)^{\sigma} j^{-\sigma(|\alpha|+|\beta|)}
		\end{equation}
	\end{linenomath*}
	for $|\alpha|>3 j, \beta \in \mathbb{Z}_{+}^{n},$ with $A, C$ positive constants independent of $\alpha, \beta, R, j$.
	
	We now define
	\begin{linenomath*}
		\[
		p(x, \xi)=\sum_{j \geqslant 0} \varphi_{j}(x, \xi) p_{j}(x, \xi).
		\]
	\end{linenomath*}
	Let us first prove that $p \in AG_{\Phi,k;\sigma}^{m_1,m_2} .$ We estimate the derivatives of $p$ in the region $Q^e_ {4R^2|\alpha|^{\sigma}}$. On the support of $\varphi_{j}$ we have $\P\japxin \geqslant 4 R^2 \sup (j^{\sigma}, 1)$. Choosing $R^2 \geqslant 2^{\sigma-2} B$
	where $B$ is the same constant appearing in Definition \ref{FS2}, for $(x,\xi) \in  Q^e_ {4R^2|\alpha|^{\sigma}}$, we have $(x,\xi) \in  Q^e_ {B(j+|\alpha|)^{\sigma}}$, then the estimates (\ref{fs1}) on the $p_{j}$ hold true. Moreover, if $\alpha\neq 0$ and $\beta \neq 0$ $D_{\xi}^{\alpha} D_{x}^{\beta} \varphi_{j}(x, \xi)$ is supported in $Q_{16 R^2 \sup \left(j^{\sigma}, 1\right)},$ then $4 R^2|\alpha|^{\sigma} \leqslant \P\japxin \leqslant 16 R^2 j^{\sigma}$, and this implies
	$|\alpha| \leqslant 4 j$. Then $(\ref{fs1}),(\ref{cut1})$ and Leibniz formula give
	\begin{linenomath*}
		\begin{align*}
			|\partial_{\xi}^{\alpha} \partial_{x}^{\beta}&(p_{j}(x, \xi) \varphi_{j}(x, \xi))|\\
			&\leqslant \sum_{\alpha^{\prime} \leqslant \alpha \atop \beta^{\prime} \leqslant \beta}
			\begin{pmatrix}
				\alpha \\
				\alpha^{\prime}
			\end{pmatrix}
			\begin{pmatrix}
				\beta \\
				\beta^{\prime}
			\end{pmatrix}
			\left|\partial_{\xi}^{\alpha^{\prime}} \partial_{x}^{\beta^{\prime}} \varphi_{j}(x, \xi) \| \partial_{\xi}^{\alpha-\alpha^{\prime}} \partial_{x}^{\beta-\beta^{\prime}} p_{j}(x, \xi)\right| \\
			&\leqslant C_{1}^{|\alpha|+|\beta|+1} \alpha !(\beta !)^{\sigma}\japxin^{m_{1}-|\alpha|-j}\P^{m_{2}-|\beta|-j}\left(\frac{C_{2}}{R}\right)^{j}(j !)^{\sigma}	\\
			& \qquad \sum_{\alpha^{\prime} \leqslant \alpha \atop \beta^{\prime} \leqslant \beta} \left(\alpha^{\prime} !\right)^{-1}\left(\beta^{\prime} !\right)^{-\sigma} j^{\left|\alpha^{\prime}\right|} j^{-\sigma\left(\left|\alpha^{\prime}\right|+\left|\beta^{\prime}\right|\right)}\japxin^{\left|\alpha^{\prime}\right|}\P^{\left|\beta^{\prime}\right|} \chi_{\operatorname{supp}(\varphi_{j})}	
		\end{align*}
	\end{linenomath*}
	where $\chi_{\text {supp }\left(\varphi_{j}\right)}$ is the characteristic function of the support of $\varphi_{j} .$ Now, if $(x, \xi) \in \operatorname{supp}\left(\varphi_{j}\right)$ we have $\japxin^{-j}\P^{-j} \leqslant\left(4 R^2 j^{\sigma}\right)^{-j}$ and $\japxin^{\left|\alpha^{\prime}\right|}\P^{\left|\beta^{\prime}\right|} \leqslant C^{\left|\alpha^{\prime}\right|+\left|\beta^{\prime}\right|} j^{\sigma\left(\left|\alpha^{\prime}\right|+\left|\beta^{\prime}\right|\right)},$ while $\left(\alpha^{\prime} !\right)^{-1} j^{\left|\alpha^{\prime}\right|} \leqslant$
	$2^{j} C^{\left|\alpha^{\prime}\right|} .$ Hence
	\begin{linenomath*}
		\[
		\left|D_{\xi}^{\alpha} D_{x}^{\beta}\left(p_{j}(x, \xi) \varphi_{j}(x, \xi)\right)\right| \leqslant C_{3}^{|\alpha|+|\beta|+1} \alpha !(\beta !)^{\sigma}\japxin^{m_{1}-|\alpha|}\P^{m_{2}-|\beta|}\left(\frac{C_{4}}{R}\right)^{j}
		\]
	\end{linenomath*}
	where $C_{4}$ is a constant independent of $R .$ Then, possibly enlarging $R$ and summing over $j$ we obtain that
	\begin{linenomath*}
		\[
		\left|D_{\xi}^{\alpha} D_{x}^{\beta} p(x, \xi)\right| \leqslant C_{5}^{|\alpha|+|\beta|+1} \alpha !(\beta !)^{\sigma}\japxin^{m_{1}-|\alpha|}\P^{m_{2}-|\beta|}
		\]
	\end{linenomath*}
	for some $C_{5}>0$ independent of $\alpha, \beta$ and for $\P\japxin \geqslant 4 R^2|\alpha|^{\sigma} .$ Similarly, using estimates (\ref{fs2}), (\ref{cut2}), we can prove that $p \in AG_{\Phi,k;\sigma, \sigma}^{m_1,m_2} .$ Then $p \in AG_{\Phi,k;\sigma}^{m_1,m_2}$. To prove that $p \sim \sum_{j \geqslant 0} p_{j}$ we observe that for $\P\japxin \geqslant 16 R^2 N^{\sigma},$ we have
	\begin{linenomath*}
		\[
		p(x, \xi)-\sum_{j<N} p_{j}(x, \xi)=\sum_{j \geqslant N} \varphi_{j}(x, \xi) p_{j}(x, \xi)
		\]
	\end{linenomath*}
	which we can estimate as before. 
\end{proof}

\begin{prop}
	Let $p \in AG^{0,0}_{\Phi,k;\sigma}$ and $\theta \geq \sigma$. If $p \sim 0$ in $FAG^{0,0}_{\Phi,k;\sigma}$, then the operator $P$ is $(\Phi,\theta)$-regularizing.
\end{prop}
The proof is similar to that of Proposition \ref{zero}. 

Using the techniques used in the calculus of $OPAG_{\Phi,k,\sigma;\sigma,\sigma}^{m_1,m_2}$, one can examine the stability of the classes $OPAG^{m_1,m_2}_{\Phi,k;\sigma}$ under transposition, composition and construction of parametrices. One can consider $\theta \geq \sigma$ and prove  Proposition \ref{trans}, Theorem \ref{comp} and Theorem \ref{ell} for the class $OPAG_{\Phi,k;\sigma}^{m_1,m_2}$.

\end{document}